\documentclass[a4paper,11pt]{amsart}
\usepackage{wrapfig}
\usepackage{ragged2e}
\usepackage{graphicx} 
\usepackage{amsmath}
\usepackage{amssymb}
\usepackage{amsthm}
\usepackage{hyperref}
\usepackage{enumerate}
\usepackage{enumitem}
\usepackage[normalem]{ulem}
\usepackage{mathdots}
\usepackage{tikz}
\usepackage{tikz-cd}
\usepackage[]{mdframed}
\usepackage[dvipsnames]{xcolor}
\usepackage{eqnarray}
\usepackage{xy}
\input xy
\xyoption{all}
\usepackage{multirow}
\usepackage{leftindex}
\usepackage{array}
\usepackage{float}
\usepackage{booktabs}
\usepackage{caption}
\usepackage{geometry}
\numberwithin{equation}{section}
\numberwithin{equation}{subsection}
\newtheorem{theorem}{Theorem}[section]
\newtheorem{proposition}[theorem]{Proposition}

\newtheorem{corollary}[theorem]{Corollary}
\newtheorem{lemma}[theorem]{Lemma}

\newtheorem*{remark*}{Remark}
\newtheorem{definition}[theorem]{Definition}
\newtheorem{remark}[theorem]{Remark}

\newtheorem{example}[theorem]{Example}

\newcommand{\C}{\mathbb{C}}
\newcommand{\F}{\mathbb{F}}

\newcommand{\Z}{\mathbb{Z}}

\newcommand{\dg}{\delta(G)}
\newcommand{\di}{\delta_{irr}(G)}
\newcommand{\tg}{\tau(G)}
\newcommand{\ti}{\tau_{irr}(G)}
\newcommand{\irr}{\mathrm{Irr}}
\newcommand{\Hom}{\mathrm{Hom}}
\newcommand{\PGL}{\mathrm{PGL}}
\newcommand{\bchi}{\Bar{\chi}}
\newcommand{\Ho}{{\mathrm{H}}}
\newcommand{\x}{\Bar{x}}
\newcommand{\y}{\Bar{y}}
\newcommand{\z}{\Bar{z}}
\newcommand{\tra}{\operatorname{tra}}

\newcommand{\GL}{\mathrm{GL}}

\newcommand{\Nayak}[1]{{\color{red} #1}}
\newcommand{\poonam}[1]{{\color{blue} #1}}
\usepackage[utf8]{inputenc}

\title[On the Faithful Projective Representations of Finite Groups]
{On the Faithful Projective Representations of Finite Groups and their Minimal Dimension}

\author{Sumana Hatui} \author{Poonam Nayak}
\address{School of Mathematical Sciences,  National Institute of Science Education and Research Bhubaneswar, 752050, Odisha, India.}
\address{Homi Bhabha National Institute, Training School Complex, Anushakti Nagar, Mumbai, 400094, India.
}
\email{sumanahatui@niser.ac.in}
\email{poonam.nayak@niser.ac.in, rashmitanayak2021@gmail.com}



\begin{document}
	
	\begin{abstract}
		The first part of this article is devoted to characterizing the cocycles $\alpha$ of a finite group $G$ that give rise to faithful projective representations of $G$. 
		We prove that a $p$-group $G$ admits a faithful irreducible projective representation if and only if the cohomology class $[\alpha]$ does not lie in the image of the inflation map
		$\inf: \mathrm{H}^2\!\left(G / N, \mathbb{C}^{\times}\right) \longrightarrow \mathrm{H}^2\!\left(G, \mathbb{C}^{\times}\right)$
		for any non-trivial central subgroup $N$ of $G$.
		In the case where $[\alpha] \in \operatorname{Im}(\operatorname{inf})$, we determine a criterion such that a direct sum of irreducible $\alpha$-representations is faithful. We conclude this part by describing the behaviour of cocycles $\alpha$ that yield faithful irreducible representations for direct products of groups.
		
		In the second part, we introduce the notion of the projective embedding degree of a finite group $G$, defined as the smallest integer $n$ such that $G$ embeds into $\PGL_n(\C)$; equivalently, it is the smallest $n$ such that $G$ has a faithful complex projective representation of degree $n$. We also define the analogous notion of the irreducible projective embedding degree of $G$. These invariants have been investigated for several classes of groups, including direct products of groups, finite abelian groups, extra-special $p$-groups, Heisenberg groups, and groups of order $p^3$, $p^4$ (for primes $p$), and $p^5$ (for $p \geq 5$). 
		
		
		
	\end{abstract}
	
	\maketitle
	
	
	\section{Introduction}
	A complex projective representation of a finite group $G$ is a homomorphism $\rho \colon G \to \PGL(V)$, where $V$ is a finite-dimensional complex vector space. Equivalently, we say a map $\rho: G \to GL(V)$ is a projective representation of $G$ if there is a map $\alpha \colon G \times G \to \mathbb{C}^\times$ such that $$\rho(g)\rho(h)= \alpha(g,h)\rho(gh), g, h \in G \text{ and } \rho(1)=1_V.$$
	Then $\alpha$ becomes a cocycle of $G$, and we refer $\rho$ as an $\alpha$-representation.
	Projective representations are a generalization of ordinary representations. The equivalence classes of projective representations of $G$ are classified by the elements of the second
	cohomology group $\mathrm{H}^2(G,\mathbb{C}^{\times})$, known as the Schur multiplier of $G$. It is also denoted by $M(G)$.
	
	A fundamental and long-standing problem in the representation theory of finite groups $G$ is to determine the smallest positive integer $n$ for which $G$ embeds into $\mathrm{GL}_n(\mathbb{C})$.
	To address this problem, the embedding degree $\delta(G)$ and irreducible embedding degree $\delta_{irr}(G)$ of a group $G$ were introduced, which are defined as follows. The embedding degree (or the representation dimension), $\delta(G)$, is the minimal dimension of a faithful complex ordinary representation of $G$, and the irreducible embedding degree $\delta_{irr}(G)$ is the minimal dimension of a faithful irreducible complex ordinary representation of $G$. 
	The representation dimension has been studied extensively for finite $p$-groups (see \cite{bardestani, cernele, kaur}) as well as for finite groups of Lie type (\cite{Lusztig, Lübeck}).
	While questions of faithfulness, irreducibility, and minimal dimension for ordinary representations of finite groups have been extensively studied in \cite{ Martino, moreto, Prajapati, SZECHTMAN}, analogous questions for projective representations have received comparatively little attention.
	Existing work on projective representations has
	primarily focused on faithfulness and irreducibility \cite{Bekka, frucht, HIGGS, Morris}, with further results obtained for specific families of groups, such as metabelian groups \cite{NG}, and finite classical groups in \cite{LANDAZURI, Zalesskii}. 
	
	These earlier studies naturally lead to the problem of determining the smallest positive integer $n$ for which a finite group $G$ embeds into $\mathrm{PGL}_n(\mathbb{C})$; equivalently, of deciding when $G$ admits a faithful (possibly irreducible) projective representation, and what the minimal possible dimension of such a representation is. 
	Let $Z^2(G,\mathbb C^\times)$ denote the set of cocycles of $G$ over $\mathbb C^\times$.
	Recall that if $N \lhd G$, the inflation map is defined as follows. For a given $\alpha \in Z^2(G/N, \mathbb{C}^{\times})$, define $\alpha' : G \times G \to \mathbb{C}^{\times}$ by $\alpha'(x, y) = \alpha(xN, yN)$. Then the assignment $[\alpha] \mapsto [\alpha']$ induces a homomorphism $\text{inf} : \text{H}^2(G/N, \mathbb{C}^{\times}) \to \text{H}^2(G, \mathbb{C}^{\times})$, called the inflation map. 
	
	Our first result appears in  Theorem \ref{cohomological characterization}, where we show that a finite $p$-group $G$ admits a faithful irreducible $\alpha$-representation if and only if the cohomology class
	$[\alpha]$ does not lie in the image of any inflation homomorphism
	\[
	\mathrm{inf} \colon
	\mathrm{H}^2\!\left(G/N, \mathbb{C}^{\times}\right)
	\longrightarrow
	\mathrm{H}^2\!\left(G, \mathbb{C}^{\times}\right),
	\]
	where $N$ ranges over all non-trivial central subgroups of $G$. As a corollary, we obtain corresponding results for non-capable groups (see Corollary \ref{non-capable groups}). Next, in Theorem \ref{direct_sum}, we determine a criterion when a direct sum of irreducible $\alpha$-representations gives rise to a faithful $\alpha$-representation for $[\alpha] \in \operatorname{Im}(\operatorname{inf})$. Furthermore, for direct products of groups $G$ with $Z(G) \subseteq [G, G]$, we provide an explicit description of the cocycles $\alpha$ that give rise to faithful irreducible $\alpha$-representations in Theorem \ref{directproduct}. 
	
	Since every finite nilpotent group is isomorphic to the direct product of its Sylow $p$-subgroups, it follows from Theorem~\ref{schur multiplier of direct product} and Theorem~\ref{tensor product is faithful} that the study of faithful irreducible projective representations of finite nilpotent groups reduces to the corresponding problem for $p$-groups. Therefore, we restrict our attention to $p$-groups.

	For an $\alpha$-representation  $\rho : G \to \mathrm{GL}(V)$ of $G$, we define the kernel of $\rho$ by $\ker \rho=\{ g \in G \mid \rho(g) = \lambda I_V$ for some $\lambda\in \C^{\times}$\}. If $\ker \rho = 1$, then we say that $\rho$ is a faithful projective representation. Note that every finite group admits a faithful complex projective representation, which follows
	from the fact that the $\alpha$-regular representation of $G$ is always faithful (see \cite[Chapter 7, Section 1]{karpilovsky}).
	Building on this observation, we introduce numerical invariants that capture the smallest possible dimensions of faithful projective representations. We define the \textit{projective embedding degree} $\tau(G)$ of \(G\) to be the smallest positive integer \(n\) for which \(G\) embeds into \(\PGL_n(\mathbb{C})\).
	Equivalently,
	\[
	\tau(G) := \min\{\deg(\rho) : \rho \text{ is a faithful complex  projective representation of } G\},
	\]
	where $\deg(\rho)$ denotes the degree of  $\rho$.
	The \textit{irreducible projective embedding degree} $\ti$ of  $G$ is defined as follows:
	\[
	\ti := \min\{\deg(\rho) : \rho \text{ is a faithful irreducible complex projective representation of } G\}.
	\] 
	We note that \(\tau_{irr}(G)\) need not exist for every group \(G\). However, whenever it exists, one has $\tau(G) \leq \tau_{irr}(G).$
	In Section~\ref{emdedding degree}, we establish sufficient conditions under which two groups \(G_1\) and \(G_2\) satisfy
	$\tau(G_1)= \tau(G_2)$
	(see Theorem~\ref{tau(G_1) = tau(G_2)}). We also study the invariants \(\tau(G)\) and \(\tau_{irr}(G)\) for several important classes of finite groups, including direct products of groups, finite abelian groups, extra-special \(p\)-groups, and Heisenberg groups.
	In Section~\ref{additional}, we present several examples of groups \(G\) and compute \(\tau(G)\) and \(\tau_{irr}(G)\). We also carry out these computations for non-abelian groups of orders \(p^3\) and \(p^4\), where \(p\) is a prime, as well as for groups of order \(p^5\) with \(p \geq 5\); see Section~\ref{groupsoforderp5}. The determination of these invariants is more intricate and requires a number of non-trivial arguments. The proofs make essential use of auxiliary lemmas and techniques developed in Sections~\ref{sec:Preliminaries}, \ref{faithful_projective}, \ref{emdedding degree}, and \ref{groupsoforderp5}.

	\begin{remark}
		One might expect that there is not much significant difference between $\tau(G)$ and $\delta(G)$; however, this is not the case. Our computations show that there is a large class of groups $G$ for which $\tau(G)$ and $\delta(G)$ differ substantially, for example, extra-special $p$-groups, Heisenberg groups, and groups of order $p^5$ in $\Phi_7$. (See Section \ref{EH}, Theorem \ref{p5} and \cite[Theorem 3.1]{kaur}.) 
	\end{remark} 
	
	
	
	
	
	Let us fix some notation here. For $x,y\in G$, the commutator $x^{-1}y^{-1}xy$ is denoted by $[x,y]$. The subgroups $G'$ and $Z(G)$ refer to the commutator subgroup and the centre of $G$, respectively. For $x \in G$, $C_G(x)$ denotes the centralizer of $x$ in $G$. The one-dimensional trivial ordinary representation of $G$ is denoted by $1_G$. We write $cd(G)= \{\deg(\rho) \mid \rho \in \irr(G) \}$, where $\mathrm{Irr}(G)$ denotes the set of all linearly inequivalent irreducible ordinary representations of $G$. If $\psi$ is an ordinary representation of $G$, then we write $K_{\psi}= \{x \in G \mid \psi(x)=I\}$. For $\alpha \in Z^2(G,\mathbb C^\times)$, $\mathrm{Irr}^\alpha(G)$ denotes the set of all irreducible $\alpha$-representations of $G$ up to linear equivalence. Furthermore, we say that a cocycle $\alpha \in Z^{2}(G,\mathbb{C}^{\times})$ is non-trivial if $[\alpha] \neq [1]$.  
	
	Throughout this paper, we restrict our attention to finite groups and complex representations. 
	By the tensor product $G_1 \otimes G_2$ of two groups $G_1$ and $G_2$, we always mean the abelian tensor product, i.e., $G_1/G_1' \otimes_{\mathbb Z} G_2/G_2'$. \\
	
	\begin{mdframed}
		\textbf{Unless stated otherwise, we say $\rho: G \to GL(V)$ is faithful on $G$ if it is faithful as a projective representation of $G$.}
	\end{mdframed}
	
	\section{Preliminaries} \label{sec:Preliminaries}
	In this section, we present several results that will be used in the proofs of subsequent sections. For basic definitions and results on projective representations, we refer to \cite[Chapter 3]{karpilovsky}.
	\begin{definition}[$\alpha$-regular element]
		For a cocycle $\alpha$ of $G$, an element $x \in G$ is called an $\alpha$-regular element of $G$ if $\alpha(g,x) =\alpha(x,g)$ for all $g\in C_G(x)$. Further, we say that $x \in G$ is an $\alpha$-regular element of $Z(G)$ if $x \in Z(G)$ and $x$ is an $\alpha$-regular element of $G$.    
	\end{definition}

	\begin{lemma}[\cite{karpilovsky}, Chapter 3, Lemma 1.1] \label{projectively equivalent to an ordinary representation}
		Suppose $\alpha_i \in Z^2(G,\mathbb C^\times)$, $i=1,2$.
		Let $\rho_1$ be an $\alpha_1$-representation. If $[\alpha_1]= [\alpha_2]$, then there exists an $\alpha_2$-representation $\rho_2$ which is projectively equivalent to $\rho_1$. In particular, if $\alpha_1$ is a coboundary, then $\rho_1$ is projectively equivalent to an ordinary representation.
		
	\end{lemma}
	
	The next result follows immediately.
	\begin{lemma} \label{projectively equivalent} 
		Let $\rho_{1}$ and $\rho_{2}$ be projective representations that are projectively equivalent. If $\rho_{1}$ is faithful, then $\rho_{2}$ is also faithful.
	\end{lemma}
	
	
	
	In view of the above results, throughout this paper, we consider projective representations associated with a fixed cocycle representative of each cohomology class. In particular, if $\alpha$ is a coboundary, then we restrict our attention to ordinary representations ($\alpha=1$).
	
	\begin{lemma} \label{one dim direct sum in non-abelian group}
		\begin{enumerate}
			\item If $G$ is a non-abelian group, then the direct sum of one-dimensional ordinary representations can never be a faithful projective representation.
			
			\item A faithful irreducible ordinary representation of $G$ is a faithful irreducible projective representation if and only if $Z(G)=1$.
			
			\item If $Z(G) = 1$, then $\tg \leq \dg $.
			Furthermore, if $Z(G) = 1$ and $M(G)=1$, then $\dg=\tg$  \text{ and }  $\ti = \delta_{irr}(G)$.
		\end{enumerate}
		
	\end{lemma}
	
	\begin{proof}
		(1) Suppose $\rho= \oplus_{k=1}^{n}\rho_{i}$, where $\rho_{i}s$ are one-dimensional ordinary representations of $G$.
		For $x \in G^\prime$, $\rho_{i}(x)=1$ and so
		$\rho(x)= I_{n \times n}$. Hence, $\rho$ cannot be faithful.
		
		(2) Let $\psi$ be a faithful irreducible ordinary representation of $G$. Then, by \cite[Chapter 2, Lemma 2.27, (f)]{isaacs}, $\{g$$|$$\psi(g)=\lambda I$ for some $\lambda\in \mathbb{C}^{\times} \} =Z(G)$.
		
		(3) Let $\psi$ be a faithful ordinary representation of $G$. Suppose $\psi(a)= \lambda I$ for some $\lambda$ $\in$ $\mathbb{C}^{\times}$ and $a \in G$. Then, $\psi(a)\psi(g)= \psi(g) \psi(a)$ $\forall g \in G$, which implies $\psi(a^{-1}g^{-1}ag)=I$. Since $\psi$ is a faithful ordinary representation, $[a,g]=1$ for all $g\in G$. 
		Hence, $a$ $\in$ $Z(G)$ and $\psi$ is a faithful projective representation. Thus, $\tg \leq \dg $. Furthermore, if $Z(G) = 1$ and $M(G)=1$, the result is obtained using $(2)$ and the argument above. 
		
	\end{proof}
	
	\begin{remark} \label{some results on p-groups} Let $G$ be a non-abelian $p$-group.
		The following facts are straightforward using the above result.
		\begin{enumerate}
			\item[(i)] $\tau(G) \geq p$.
			
			\item[(ii)] $\tau(G)=p$ if and only if $\tau_{irr}(G) =p$. In this case, the corresponding faithful representation must be a true projective representation.
			
			\item[(iii)] If $\tau_{irr}(G)$ does not exist or $\tau_{irr}(G)>p$, then $\tau(G) > p$.
		\end{enumerate}
	\end{remark}
	
	
	
	
	Let $N$ be a normal subgroup of $G$ and  $\rho \in \operatorname{Irr}(N)$. The induced representation of $\rho$ from $N$ to $G$ is denoted by $\mathrm{Ind}_N^G(\rho)$. We say $\rho \in \mathrm{Irr}(N)$  extends to $G$ if there exists $\tilde{\rho} \in \mathrm{Irr}(G)$ such that $\tilde{\rho}|_N = \rho$.
	Let $\chi_\rho$ be the character affording $\rho$. For $g \in G$, define conjugate character $\chi_\rho^{(g)}: N \to \C$ by $\chi_\rho^{(g)}(x) = \chi_\rho(gxg^{-1})$ $\forall$ $x \in N$. Then, 
	\[
	I_G(\chi_\rho) = \{ g \in G \mid \chi_\rho^{(g)} = \chi_\rho \}
	\]
	is called the inertia group of $\rho$ in $G$. For any character $\chi$ of $N$, we define  $$\operatorname{Irr}(G \mid \chi)= \{\eta \in \irr(G) \mid \langle \chi_{\eta\mid_N}, \chi \rangle \neq 0\}$$ denote the set of inequivalent irreducible ordinary representations of $G$ lying above $\chi$. 
	Next, we recall a few results of Clifford's theory.
	\begin{theorem}[\cite{isaacs}, Theorem 6.11, Corollary 11.22] Let $N \triangleleft G$ and $\rho \in \mathrm{Irr}(N)$. If $\chi$ is the character affording $\rho$, then the following hold.
		
		\begin{enumerate} \label{Clifford's theory}
			
			\item The map 
			\[
			\theta \mapsto \operatorname{Ind}^{G}_{I_G(\chi)}(\theta)
			\] is a bijection of $\operatorname{Irr}(I_G(\chi) \mid \chi)$ onto $\operatorname{Irr}(G \mid \chi)$.
			
			\item Suppose $G/N$ is cyclic and $I_G(\chi) = G$. Then $\rho$ extends  to $G$.
			
		\end{enumerate}    
	\end{theorem}
	
	Next, we recall the following definition from \cite{gonzalo}. 
	\begin{definition} \label{covering group}
		Let $G$ be a group. A group $A.G$ is called a covering of $G$ if there exists a central extension
		\[
		1 \to A \to A.G \to G \to 1
		\]
		such that $A \subseteq M(G)$ and $A \subseteq (A.G)'$. 
		Furthermore, if $A \cong M(G)$, then $A.G$ is called a representation group of $G$, and we denote it by  $G^*$.
	\end{definition} 
	Let $1 \to A \to H \xrightarrow{f} G \to 1$ be a central extension and $\mu$ be a section of $f$, i.e., $\mu: G \to H$ is a map such that $f \circ \mu =id_G$ and $\mu(1)= 1$. 
	For a $\chi \in \mathrm{Hom}(A, \mathbb{C}^{\times})$, define 
	$\alpha_\chi : G \times G \to \mathbb{C}^{\times}$ by $\alpha_\chi(x, y) = \chi(\mu(x)\mu(y)\mu(xy)^{-1})$, $x,y \in G$. Then the assignment $\chi \mapsto [\alpha_\chi]$ induces a homomorphism $\mathrm{tra}: \mathrm{Hom}(A, \mathbb{C}^{\times}) \to \mathrm{H}^2(G, \mathbb{C}^{\times})$, called the transgression map associated with the given central extension. For simplicity, we shall denote the representative cocycle $\alpha_\chi$ by $\mathrm{tra}(\chi)$. 
	
	Now, the result below follows from \cite[Chapter 2, Theorem 2.5]{karpilovsky} and \cite[Proposition 2.1.7]{schurmultiplier}.
	\begin{lemma} \label{exact sequence}
		Let $A \subseteq G' \cap Z(G)$ and $1 \to A \to G \to G/A \to 1$ be the natural exact sequence. 
		Then the induced sequence
		
		$$ 1  \xrightarrow{}\mathrm{Hom}(A,\mathbb{C}^{\times}) 
		\xrightarrow{\mathrm{tra}} 
		\mathrm{H}^2(G/A,\mathbb{C}^{\times}) 
		\xrightarrow{\mathrm{inf}} \mathrm{H}^2(G,\mathbb{C}^{\times})$$
		is exact.
		
	\end{lemma}
	
	The next result establishes a relationship between the projective representations of a group $G$ and those of a suitable quotient group of $G$,
	which follows from the proof of \cite[Theorem 3.2]{hatui}.
	\begin{theorem} \label{bijective correspondence via inflation}
		Let $A$ be a subgroup of a finitely generated group $G$ such that $A \subseteq G' \cap Z(G)$ and $[\alpha] \in \mathrm { Im } \big(\mathrm{inf}: \mathrm{H}^2(G/A, \C^{\times}) \to \mathrm{H}^2(G, \C^{\times})\big)$. Suppose $\rho : G \to \mathrm{GL}(V) \in \mathrm{Irr}^\alpha(G)$.
		Define $\tilde{\rho} : G/A \to \mathrm{GL}(V)$ by 
		$\tilde{\rho}(gA) = \rho\big(\mu(gA)\big)$, where $\mu : G/A \to G$ is a section of $G \twoheadrightarrow G/A$.
		Then the map
		\[
		\phi : \mathrm{Irr}^\alpha(G) \longrightarrow 
		\bigcup_{\{[\beta] \in \mathrm{H}^2(G/A, \mathbb{C}^{\times}) \mid \mathrm{inf}([\beta]) = [\alpha]\}} 
		\mathrm{Irr}^\beta(G/A)
		\]
		given by $\phi(\rho) = \tilde{\rho}$ is a dimension-preserving bijective map. 
	\end{theorem} 
	
	\begin{remark} \label{remark for bijective correspondence via inflation}
		A closer look at the map $\phi$  gives the following.
		\begin{enumerate}
			\item  If \(\beta \in Z^{2}(G/A,\C^{\times})\) and $\alpha(g,h) = \beta(gA,hA), g,h \in G$, then $\phi$ gives a bijective correspondence between \(\mathrm{Irr}^{\alpha}(G)\) and
			$\bigcup_{\chi \in \mathrm{Hom}(A,\C^{\times})}
			\mathrm{Irr}^{\,\beta\mathrm{tra}(\chi)}(G/A)$.
			
			\item If $\tilde{\rho} \in \operatorname{Irr}^{\beta \mathrm{tra}(\chi)}(G/A)$, then $\rho : G \to \mathrm{GL}_n(\C)$ defined by
			\[
			\rho(a \mu(gA)) = \chi(a)\tilde{\rho}(gA) \quad \text{for all } a \in A, g \in G  
			\] is such that $\phi(\rho) = \tilde{\rho}$ and $\rho |_A=\chi$.
		\end{enumerate}
	\end{remark}
	We now recall the fact that every projective representation of $G$ lifts to an ordinary representation of $G^{*}$, where $G^*$ is a representation group of $G$. Let 
	$$1 \longrightarrow A \longrightarrow G^* \xrightarrow{f} G \longrightarrow 1$$
	be the corresponding central extension and $\mu$ be a section of $f$. Then the associated transgression map
	$
	\operatorname{tra} : \operatorname{Hom}(A, \mathbb{C}^{\times}) \longrightarrow 
	\mathrm{H}^{2}(G, \mathbb{C}^{\times})
	$
	is an isomorphism (see \cite [Theorem 2.1.4]{schurmultiplier}).
	Suppose $\chi \in \mathrm{Hom}(A, \mathbb{C}^{\times})$ and  $\alpha= \mathrm{tra}(\chi)$. For a given $\Gamma \in \operatorname{Irr}(G^* \mid \chi)$, let $\rho : G \to \mathrm{GL}_n(\C)$ be defined by
	\[
	\rho(g) = \Gamma(\mu(g)) \quad \text{for all }g \in G.
	\]
	It is easy to see that $\rho \in \operatorname{Irr}^{\alpha}(G)$; in this case, we say $\rho$ lifts to $\Gamma$.
	The next result follows from the proof of \cite[Chapter 4, Lemma 1.2]{karpilovsky3}. 
	
	\begin{theorem} \label{bijective correspondence}
		With the above notations, for $\alpha = \operatorname{tra}(\chi)$, the sets $\operatorname{Irr}(G^{*} \mid \chi)$ and $\operatorname{Irr}^{\alpha}(G)$ are in dimension-preserving bijective correspondence, given by the map
		\[
		\Gamma \longmapsto \rho.
		\]
		In particular, the sets
		\[
		\operatorname{Irr}(G^{*})
		\quad \text{and} \quad \bigcup_{[\alpha] \in \mathrm{H}^{2}(G, \mathbb{C}^{\times})}
		\operatorname{Irr}^{\alpha}(G)
		\] are in dimension-preserving bijective correspondence. 
	\end{theorem}

	We now summarize the following results from  \cite[Chapter 10, Lemma 4.3, Lemma 4.12]{karpilovsky}.
	
	\begin{lemma} \label{faithful projective representation}
		
		
		Let $\alpha \in Z^2(G,\mathbb{C}^{\times})$.
		\begin{enumerate}
			\item If $\rho \in \operatorname{Irr}^{\alpha}(G)$ is faithful, then $1$ is the only $\alpha$-regular element of $Z(G)$.
			
			\item Suppose that $G$ is nilpotent. If $1$ is the only $\alpha$-regular element of $Z(G)$, then every $\alpha$-representation of $G$ is faithful.
			
		\end{enumerate}
	\end{lemma}
	
	\begin{lemma} \label{restriction}
		Let $G$ be a nilpotent group. Suppose that $\alpha \in Z^2(G,\mathbb{C}^{\times})$ and $\rho \in \mathrm{Irr}^\alpha(G)$. Then, the following hold:
		\begin{enumerate}
			\item If $\rho$ is faithful, then any element of $\mathrm{Irr}^\alpha(G)$ is faithful.
			\item $\rho$ is faithful if and only if $\rho|_{Z(G)}$ is faithful. 
		\end{enumerate}
		
	\end{lemma}
	
	\begin{proof}
		(1) follows  from Lemma \ref{faithful projective representation}. 
		
		(2) If ker $\rho \neq 1$, then the result follows by the fact that ker $\rho$ intersects $Z(G)$ non-trivially (see \cite[Chapter 10, Lemma 1.7]{karpilovsky2}). 
		
	\end{proof}
	
	
	\begin{corollary} \label{restriction on A is faithful}
		Let $G$ be a $p$-group such that $Z(G)$ is cyclic. Suppose that $\alpha \in Z^2(G,\mathbb{C}^{\times})$ and $\rho \in \mathrm{Irr}^\alpha(G)$. Then, $\rho$ is faithful if and only if $\rho|_A$ is faithful for any non-trivial central subgroup $A$ of $G$.   
	\end{corollary}
	
	\begin{proof}
		If $\rho$ is faithful, then clearly $\rho|_A$ is faithful. Now, let $Z(G)= \langle x \rangle$ and $A= \langle x^{p^m} \rangle$. Assume that $\rho|_A$ is faithful. If $\rho|_{Z(G)}$ is not faithful, then $\exists$ $g \in Z(G)$ such that $\rho(g)= \lambda I$ for some $\lambda \in \C^\times$. Let $g= x^{kp^n}$ for some $1 \leq k \leq p-1$. If $m \leq n$, then $g \in A$, which leads to a contradiction. Suppose $m> n$. As $Z(G)=\langle x \rangle= \langle x^k \rangle$, there is a $r>0$ such that  $\rho(x^{p^m})= \rho(x^{krp^m})= \lambda_1 I$ for some $\lambda_1 \in \C^\times$, a contradiction.  
		Hence, $\rho|_{Z(G)}$ is faithful. The result now follows from Lemma \ref{restriction}.
	\end{proof}

	\begin{lemma}[\cite{karpilovsky3}, Chapter 8, Theorem 2.21] \label{properties of abelian groups}
		Let $G$ be a finite abelian group. Suppose $\alpha \in Z^2(G,\mathbb C^\times)$ and $G_0$ is the subgroup of $G$ consisting of all $\alpha$-regular elements of $G$. Then $\alpha|_{G_0 \times G_0}$ is a coboundary and $\operatorname{deg}(\rho) = \sqrt{|G/G_0|}$ for any $\rho \in \mathrm{Irr}^\alpha(G)$.
	\end{lemma}  
	
	Frucht \cite{frucht} showed that a finite abelian group $G$ has a faithful irreducible projective representation if and only if $G$ is of symmetric type, i.e., $G$ can be written as a direct product of two isomorphic subgroups. 
	
	\begin{theorem} \label{projective irreducible embedding degree of finite abelian groups}
		Let $G$ be a finite abelian group of symmetric type. Then $\tau_{irr}(G) = \sqrt{|G|}$.    
	\end{theorem}
	
	\begin{proof}
		This is a consequence of \cite[Chapter 10, Lemma 4.5]{karpilovsky}.     
	\end{proof}
	
	We now state some results on capable groups. A group $\Gamma$ is called capable if there exists a group $\triangle$ with $\Gamma \cong \triangle/Z(\triangle)$. The epicentre of $G$ is denoted by $Z^{*}(G)$, which is the smallest central subgroup of $G$ such that $G/Z^*(G)$ is capable.
	
	\begin{theorem}[\cite{schurmultiplier}, Theorem 2.5.10] \label{epicenter} Let $Z$ be a central subgroup of $G$. Then the following conditions are equivalent:
		\begin{enumerate}
			\item[(a)] $Z \subseteq Z^*(G)$
			\item[(b)] $\mathrm{inf} : \mathrm{H}^2(G/Z,\mathbb{C}^{\times})\rightarrow\mathrm{H}^2(G,\mathbb{C}^{\times})$ is surjective.   
		\end{enumerate}
	\end{theorem}
	
	Recall that a group $G$ is said to be a central product of its normal subgroups $H$ and $K$ if $G=HK$, $[H, K]= 1$, and $H \cap K \neq 1$.
	
	\begin{lemma} \label{central product}
		Let $G$ be a central product of its subgroups $H$ and $K$ with $H^{\prime} \cap K^{\prime} \neq 1$. Then $G$ is not capable.      
	\end{lemma}
	
	\begin{proof}
		Let $Z = H^{\prime}$ $\cap$ $K^{\prime}$. By \cite[Theorem 3.1]{central}, inf : $\mathrm{H}^2\left(G / Z, \mathbb{C}^{\times}\right) \rightarrow \mathrm{H}^2\left(G, \mathbb{C}^{\times}\right)$ is a surjective map. Thus, the proof follows from Theorem \ref{epicenter}. 
		
	\end{proof}
	The next result follows from the proof of \cite[Chapter 4, Theorem 3.1]{karpilovsky2}.
	\begin{theorem} \label{capable groups}
		Let $G$ be a group. The following properties are equivalent:
		\begin{enumerate}
			\item[(i)] $G$ affords an irreducible faithful projective representation;
			\item[(ii)] there exists a group $H$ admitting an irreducible faithful ordinary representation such that $H/Z(H) \cong G$ .   
		\end{enumerate}
		Further, if $\delta_{irr}(H)$ exists, then $\tau_{irr}(G)$ also exists and satisfies $\tau_{irr}(G) \le \delta_{irr}(H)$. 
		
	\end{theorem}
	
	\begin{lemma}[\cite{Beyl}, Corollary 8.2] \label{capability of extra special $p$-group}
		An extra-special $p$-group is capable if and only if it is either a dihedral group of order $8$ or of order $p^3$ having exponent $p > 2$.    
	\end{lemma}   
	
	
	
	\section{Results on Faithful Projective Representations of $G$} \label{faithful_projective}
	Our first main result provides a cohomological characterization of the existence of faithful irreducible projective representations.
	
	\begin{theorem} \label{cohomological characterization}
		A p-group $G$ has a faithful irreducible $\alpha$-representation if and only if $[\alpha] \notin \mathrm {Im}\big(\operatorname{inf} : \mathrm{H}^2\left(G / N, \mathbb{C}^{\times}\right) \rightarrow \mathrm{H}^2\left(G, \mathbb{C}^{\times}\right)\big)$ for any non-trivial central subgroup $N$ of $G$.    
	\end{theorem}
	
	\begin{proof}
		Let $\rho \in \mathrm{Irr}^\alpha(G)$ be not faithful. Then, for $N= \operatorname{ker}(\rho) \cap Z(G)$, we have the exact sequence
		$$
		\mathrm{H}^2(G / N, \mathbb{C}^{\times}) \xrightarrow{\text {inf}} \mathrm{H}^2(G, \mathbb{C}^{\times}) \xrightarrow{\chi} \mathrm{H}^2(N, \mathbb{C}^{\times}) \oplus \mathrm{Hom}(G \otimes N, \mathbb{C}^{\times})
		$$
		where $\chi=(\operatorname{res}, \psi)$ as defined by Iwahori and Matsumoto \cite{matsumoto}. To be more precise, $\operatorname{res}: \mathrm{H}^2(G, \mathbb{C}^{\times}) \rightarrow \mathrm{H}^2(N, \mathbb{C}^{\times})$ is the restriction homomorphism and $\psi: \mathrm{H}^2(G, \mathbb{C}^{\times}) \rightarrow \mathrm{Hom}(G \otimes N, \mathbb{C}^{\times})$ is defined as $\psi(\xi)(\overline{x}, n)=f(x, n)f(n, x)^{-1}$ for all $\overline{x}=x G^{\prime} \in G / G^{\prime}$ and $n \in N$, where $\xi\in \mathrm{H}^2(G, \mathbb{C}^{\times})$ and $f$ is a cocycle representative of $\xi$. 
		Since $\alpha|_{\operatorname{ker}(\rho) \times \operatorname{ker}(\rho)}$ is a coboundary,  $\operatorname{res}([\alpha])$ is trivial. We also have $\alpha(x, n)= \alpha(n, x)$ $\forall$ $n$ $\in$ $N$, $x\in G$. Therefore, $[\alpha] \in \operatorname{ker}(\psi)$. Thus, $[\alpha] \in \mathrm{ker}(\operatorname{\chi}) = \mathrm{Im} (\operatorname{inf})$. 
		
		Conversely, let $N$ be a non-trivial central subgroup of $G$, $[\alpha] = \operatorname{inf}([\beta])$ for some $[\beta] \in \operatorname{H}^2(G/N, \mathbb{C}^{\times})$ and $\rho\in \mathrm{ Irr}^\alpha(G)$. 
		Then $\alpha(g,h)= \beta(gN,hN)$ $\forall$ $g,h \in G$. If $a \in N$, then $\alpha(a,g) = \alpha(g,a)= 1$ $\forall$ $g \in G$ and so $a$ is an $\alpha$-regular element of $Z(G)$. Hence, $\rho$ cannot be faithful, due to Lemma \ref{faithful projective representation}$(1)$.
		
	\end{proof}
	
	
	\begin{corollary} \cite[Chapter 4, Theorem 3.1]{karpilovsky2} \label{non-capable groups}
		Let $G$ be a non-capable group. Then, $G$ does not have a faithful irreducible projective representation.    
	\end{corollary}
	
	\begin{proof}
		By Theorem \ref{epicenter}, inf : $\mathrm{H}^2(G/Z^*(G),\mathbb{C}^{\times})\rightarrow \mathrm{H}^2(G,\mathbb{C}^{\times})$ is surjective. Therefore, the result follows by the same argument used in the second paragraph of the proof of Theorem \ref{cohomological characterization}. 
		
	\end{proof}
	
	
	Let $G$ be a $p$-group with a subgroup $A$ such that $A \subseteq Z(G) \cap G^\prime$. If $[\alpha] \in \mathrm{Im} (\mathrm{inf})$, then $\exists$ $\beta \in \mathrm{Z}^2(G/A, \C^{\times})$ such that $\alpha(g,h) = \beta(gA,hA)$ $\forall g,h \in G$. Recall the map $\phi$ from Theorem \ref{bijective correspondence via inflation} which gives a bijection between $\mathrm{Irr}^\alpha(G)$ and $\bigcup_{\{[\beta] \in \mathrm{H}^2(G/A, \mathbb{C}^{\times}) \mid \mathrm{inf}([\beta]) = [\alpha]\}} 
	\mathrm{Irr}^\beta(G/A)$. Under these assumptions and notation, we have the following result.
	\begin{theorem} \label{direct_sum} 
		Suppose $\rho_i \in \mathrm{Irr}^\alpha(G)$ and  $\phi(\rho_i)= \tilde\rho_i$ for $i=1,...,k$. Then $\tilde\rho_i$ is a $\beta\mathrm{tra}(\chi_i)$-representation of $G/A$ for some $\chi_i\in \mathrm{Hom}(A,\mathbb{C}^{\times})$.
		Furthermore, if one of the following conditions is satisfied, 
		\begin{enumerate}
			\item $Z(G)$ is cyclic,
			\item $\big(\cap_{i=1}^k \operatorname{ker}(\tilde\rho_i)\big)\cap \big(Z(G)/A \big)= 1$,
		\end{enumerate}
		then $\oplus_{i=1}^k \rho_i$ is faithful on $G$ if and only if $\oplus_{i=1}^k \chi_i$ is faithful on $A$.
		
	\end{theorem}
	
	
	\begin{proof}
		By Remark \ref{remark for bijective correspondence via inflation}$(1)$, $\tilde\rho_i$ is a $\beta\mathrm{tra}(\chi_i)$ representation of $G/A$ for some $\chi_i \in \mathrm{Hom}(A,\mathbb{C}^{\times})$.
		\item 
		\begin{enumerate}
			\item It is easy to see that $\chi_i= \rho_i|_A$. The result now follows from Corollary \ref{restriction on A is faithful}. 
			\item 
			Consider the central extension
			
			\[
			1 \to A \to Z(G) \to Z(G)/A \to 1
			\]
			Let $\mu: Z(G)/A \to Z(G)$ be a section of $Z(G) \twoheadrightarrow Z(G)/A$.
			Let $x \in Z(G)$ such that $(\oplus_{i=1}^k \rho_i)(x)=\lambda I$ for some $\lambda \in \mathbb{C}^{\times}$. Then $x= a\mu(gA)$
			for some $gA \in Z(G)/A$, $a\in A$ and by Remark \ref{remark for bijective correspondence via inflation}(2), $\chi_i(a) \tilde\rho_i(gA) = \lambda I$ $\forall$ $1 \leq i \leq k$. From the given hypothesis, it follows that $g \in A$ and so $x \in A$. Thus, $\oplus_{i=1}^k \rho_i$ is faithful on $Z(G)$ if and only if $\oplus_{i=1}^k \chi_i$ is a faithful on $A$. Hence, the result follows from Lemma \ref{restriction}$(2)$. 
			
		\end{enumerate}
	\end{proof}
	
	\subsection{Direct products} Our next aim is to describe the faithful irreducible projective representations of the direct product of groups. 
	For direct products of groups $G$ satisfying $Z(G) \subseteq [G, G]$, we characterize cocycles $\alpha$ for which the corresponding irreducible $\alpha$-representations are faithful.   
	\begin{theorem}[\cite{karpilovsky}, Chapter 2, Theorem 3.13] \label{schur multiplier of direct product}
		If $G = H \times K$, then 
		$$
		\mathrm{H}^2(G, \mathbb{C}^{\times}) \cong\mathrm{H}^2(H, \mathbb{C}^{\times}) \times \mathrm{H}^2(K, \mathbb{C}^{\times}) \times \mathrm{Hom}(H \otimes K, \mathbb{C}^{\times}).$$ 
	\end{theorem}
	The above isomorphism $\mathrm{H}^2(H, \mathbb{C}^{\times}) \times \mathrm{H}^2(K, \mathbb{C}^{\times}) \times \mathrm{Hom}(H \otimes K, \mathbb{C}^{\times}) \to \mathrm{H}^2(H \times K, \mathbb{C}^{\times})$ is given 
	by the map
	\[
	([\beta],[\gamma],f) \longmapsto [\alpha], 
	\]
	such that $\alpha(h_1k_1,h_2k_2)= \beta(h_1,h_2)\,\gamma(k_1,k_2)\,f(\overline{h_1}\otimes\overline{k_2}), $ for $ h_1,h_2 \in H, $  $k_1,k_2 \in K, \overline{h_1}= h_1H', \overline{k_2}= k_2K'$.
	Via this isomorphism, we write $\alpha=(\beta,\gamma,f)$. \\
	

	Recall that, for two matrices $A=(a_{ij})_{m\times m}$ and $B=(b_{ij})_{n\times n}$, their tensor product is defined by
	\[
	A \otimes B
	=
	\begin{pmatrix}
		a_{11}B & \cdots & a_{1m}B \\
		\vdots  & \ddots & \vdots \\
		a_{m1}B & \cdots & a_{mm}B
	\end{pmatrix}.
	\]
	The next result is immediate.
	\begin{lemma} \label{tensor product of matrices}
		Let $A \in GL_m(\mathbb{C})$ and $B \in GL_n(\mathbb{C})$. Then $A \otimes B = \lambda I_{mn}$ for some $\lambda \in \mathbb{C}^\times$ if and only if there exist scalars $\lambda_1, \lambda_2 \in \mathbb{C}^\times$ such that $A = \lambda_1 I_m \text{ and } B = \lambda_2 I_n.$
	\end{lemma}
	
	Let $\rho_1: H \to GL(V_1)$ be a $\beta$-representation of $H$ and $\rho_2: K \to GL(V_2)$ be a $\gamma$-representation of $K$. The tensor product $\rho_1 \otimes \rho_2: H \times K \to GL(V_1\otimes V_2)$ is defined by $$(\rho_1 \otimes \rho_2) (h,k)= \rho_1(h) \otimes \rho_2(k), h \in H, k \in K.$$ It is easy to check that $\rho_1 \otimes \rho_2$ is a $(\beta, \gamma, 1)$-representation of $G$
	(\cite[Chapter 5, Section 1]{karpilovsky}).
	
	
	\begin{theorem} \label{tensor product is faithful} 
		Let $G = H \times K$ and $\alpha = (\beta,\gamma,1)$ $\in$ $Z^{2}(G,\mathbb{C}^{\times})$. Then, up to linear equivalence, an irreducible $\alpha$-representation $\rho$ is of the form $\rho_1 \otimes \rho_2$ for some $\rho_1\in \mathrm{Irr}^\beta(H), \rho_2\in \mathrm{Irr}^\gamma(K)$. Furthermore,
		$\rho$ is faithful on $G$ if and only $\rho_1$ and $\rho_2$ are faithful on $H$ and $K$ respectively.
	\end{theorem}
	
	\begin{proof}
		By \cite[Chapter 5, Corollary 1.3]{karpilovsky}, it follows that, up to linear equivalence, $\rho= \rho_1 \otimes \rho_2$ for some $\rho_1 \in \mathrm{Irr}^\beta(H)$ and $\rho_2 \in \mathrm{Irr}^\gamma(K)$. 
		
		
		Suppose $\rho= \rho_1 \otimes \rho_2$ is faithful and $\rho_2$ is not faithful.
		Then, $\exists$ $k^{\prime} \in K$, $k^{\prime} \neq 1$ such that $\rho_{2}(k')= d I_{n \times n}$ for some $d \in \mathbb{C}^{\times}$. Then, we have $$(\rho_1 \otimes \rho_{2})\left(1,k^{\prime}\right)= \rho_{1}\left(1\right) \otimes \rho_{2}(k') = I_{m\times m}  \otimes d I_{n \times n}= d I_{m n \times m n}.$$ 
		This contradicts our assumption that $\rho_{1} \otimes \rho_{2}$ is faithful. So, $\rho_{2}$ must be faithful. Using similar arguments, one can verify that $\rho_{1}$ is also faithful.
		Conversely, let $\rho_1$ and $\rho_2$ be faithful.
		If $(\rho_{1} \otimes \rho_{2})(h,k)=$ $\lambda I$ for some $\lambda \in \mathbb{C}^{\times}$, then $\rho_{1}(h) \otimes \rho_{2}(k) = \lambda I$. Hence, the result follows from Lemma \ref{tensor product of matrices}. 
	\end{proof}
	
	\begin{theorem} \label{directproduct}
		Let $G = H \times K$, where $H$ and $K$ are non-abelian $p$-groups. Let $\alpha$ $\in$ $Z^{2}(G,\mathbb{C}^{\times})$. 
		\begin{enumerate}
			\item If $\alpha=(1,1,f)$, then $\rho \in \mathrm{Irr}^\alpha(G)$ is not faithful.
			
			\item If $Z(G) \subseteq G^{\prime}$ and $\alpha = (\beta,\gamma,f)$, then an irreducible $\alpha$-representation of $G$ is faithful if and only if $1$ is the only $\beta$-regular element of $Z(H)$ and $1$ is the only $\gamma$-regular element of $Z(K)$. Furthermore, $G$ admits a faithful irreducible $\alpha$-representation if and only if $H$ and $K$ admit faithful irreducible $\beta$ and $\gamma$-representation respectively.     
		\end{enumerate}
	\end{theorem}
	
	\begin{proof}
		(1) Let $N = \big(H^{\prime}$ $\cap$ $Z(H)\big) \times \big(K^{\prime}$ $\cap$ $Z(K)\big)$. Since $\alpha|_{N \times G} = \alpha|_{G \times N}= 1$, there exists a $\beta\in \mathrm{Z}^2(G/N, \C^{\times})$ such that $\operatorname{inf}([\beta])=[\alpha]$.  
		The result now follows from Theorem \ref{cohomological characterization}. 
		\\  
		
		(2) By Lemma \ref{faithful projective representation}, it suffices to show that $1$ is the only $\alpha$-regular element of $Z(G)$ if and only if $1$ is the only $\beta$-regular element of $Z(H)$ and $1$ is the only $\gamma$-regular element of $Z(K)$.    
		Let $(h,k)$ be an $\alpha$-regular element of $Z(G)$. Then, $\alpha \big((h,k),(h_1,k_1)\big)= \alpha\big((h_1,k_1),(h,k)\big)$ $\forall$ $(h_1,k_1) \in G$. Putting $k_1 = 1$, we deduce that $\beta(h,h_1) = \beta(h_1,h)$ $\forall$ $h_1 \in H$. Hence, $h$ is a $\beta$-regular element of $Z(H)$. Similarly, it can be shown that $k$ is a $\gamma$-regular element of $Z(K)$.
		Conversely, let $h$ be a $\beta$-regular element of $Z(H)$ and $k$ be a $\gamma$-regular element of $Z(K)$. It is easy to check that $\alpha\big((h,k),(h_1,k_1)\big)= \alpha\big((h_1,k_1),(h,k)\big)$ $\forall$ $(h_1,k_1)\in G$. Therefore, $(h,k)$ is an $\alpha$-regular element of $Z(G)$. This completes the proof of the result. 
	\end{proof}
	
	\section{$\tau(G)$ and $\tau_{irr}(G)$ for finite groups $G$} \label{emdedding degree} We first establish a relationship between $\delta(G)$ and $\tau(G)$. 
	
	
	\begin{lemma} \label{relation}
		For any finite group $G$, $\tau(G) \le \delta(G)+ 1$.
	\end{lemma}
	
	\begin{proof} Let $\delta(G)= n$ and $\psi$ be a faithful ordinary representation of $G$ of degree $n$. 
		Suppose $\rho$: $G\,\to\,GL_{n+1}(\mathbb{C})$ is defined by \[\rho(g)=
		\begin{bmatrix}
			\psi(g) & 0 \\
			0 &  1
		\end{bmatrix}.
		\]
		Observe that $\rho$ is a faithful projective representation of $G$. Hence, the result follows.
	\end{proof}
	Note that the above bound is optimal, as is evident from the next theorem. 
	\begin{lemma} \label{group with cyclic centre of order p}
		Let $G$ be a $p$-group such that $M(G)= 1$ and $Z(G)$ is of order $p$. Then $\tau_{irr}(G)$ does not exist and $\tau(G)= \delta(G)+1$.
	\end{lemma}
	
	\begin{proof}
		
		Since $M(G)=1$, by Lemma \ref{projectively equivalent to an ordinary representation} and Lemma \ref{projectively equivalent}, it suffices to consider only ordinary representations of $G$.
		Since faithful projective representations are also faithful ordinary representations, by Lemma \ref{relation}, we have $\dg \leq \tau(G) \leq \dg+1$. Let $\tau(G)= \dg = m$ and $\rho$ be an ordinary representation of degree $m$ which is faithful as a projective representation of $G$. We write $\rho = \oplus_{i=1}^k \rho_i$, where $\rho_i \in \mathrm{Irr}(G)$. By Lemma \ref{one dim direct sum in non-abelian group}$(2)$, it follows that $\tau_{irr}(G)$ does not exist. Hence, $k > 1$. 
		Since $\rho_i$ is not a faithful ordinary representation, by the fact used in the proof of Lemma \ref{restriction}, $\rho_i|_{Z(G)}$ cannot be a faithful ordinary representation. Thus, $\rho_i|_{Z(G)}= n_i 1_{Z(G)}$ as $Z(G)$ is of order $p$. Therefore, $\rho|_{Z(G)}= n 1_{Z(G)}$ for some $n \in \mathbb N$. Since $\rho|_{Z(G)}$ is not faithful, by Lemma \ref{restriction}, $\rho$ cannot be a faithful on $G$. Consequently, $\tau(G)= \delta(G)+1$. 
		
	\end{proof}
	In the next result, we describe $\tau_{irr}(G)$ and $\tg$ for $G= H \times K$. 
	
	
	
	\begin{theorem} \label{projective irreducible embedding degree for direct product} Let $G = H \times K$. \\
		(i) If $gcd(|H|,|K|)=1$, then \textbf{$\tau_{irr}(G) = \tau_{irr}(H) \tau_{irr}(K)$}.   \\
		(ii) If $M(H)= 1$ and $M(K)= 1$, then \textbf{$\tau(G) \leq \tau(H) + \tau(K)$}. \\
		(iii) \textbf{$\tau(G) \leq \tau(H) \tau(K)$}.    
	\end{theorem}
	
	\begin{proof} 
		(i) Let $\tau_{{irr}}(H)=m$,  $\tau_{irr}(K)=n$ and $\tau_{irr}(G)=d$. 
		Since $\gcd(|H|, |K|) = 1$, by Theorem~\ref{schur multiplier of direct product}, any cocycle $\alpha \in Z^2(G,\mathbb{C}^\times)$ is of the form $\alpha = (\beta, \gamma, 1).$ 
		Hence, by Theorem \ref{tensor product is faithful}, any $\rho \in \mathrm{Irr}^\alpha(G)$ is of the form $\rho = \rho_{1} \otimes \rho_{2}$, where $\rho_{1} \in \operatorname{Irr}^{\beta}(H), \rho_{2} \in \operatorname{Irr}^{\gamma}(K)$, and
		$\rho$ is faithful if and only if $\rho_{1}, \rho_{2}$ are also faithful. Since 
		deg $(\rho_{1} \otimes \rho_{2}) =(\operatorname{deg} \rho_{1})(\operatorname{deg} \rho_{2}), ~d = mn$ and the result follows.
		
		
		(ii) Let $\tau(H)= m$ and $\tau(K)= n$. Let $\eta$ and $\psi$ be faithful $\beta$ and $\gamma$-representation of $H$ and $K$ of degree $m$ and $n$ respectively. Since $M(H)=1$ and $M(K)=1$, we can assume that $\beta= \gamma=1$. Define $\sigma: G \to GL_{m+n}(\mathbb{C})$ by:
		\[
		\sigma(h,k)= \begin{bmatrix}
			\eta(h) &  \\
			& \psi(k)
		\end{bmatrix}
		\]
		It is easy to verify that $\sigma$ is a faithful $\alpha$-representation of $G$ for $\alpha =(1,1,1)$. Thus, $\tau(G) \leq \tau(H) + \tau(K)$.
		
		(iii) Let $\eta$ and $\psi$ be faithful $\beta$ and $\gamma$-representation of $H$ and $K$ respectively. Using similar arguments as in the proof of Theorem \ref{tensor product is faithful}, $\eta \otimes \psi$ is faithful. Consequently, the result follows.   
	\end{proof}

	\subsection{Non-capable groups $G_1$ and $G_2$ for which $\tau(G_1)= \tau(G_2)$}
	In the next result, we determine sufficient conditions under which non-capable groups have the same projective embedding degree.
	We say that an isomorphism $\psi : G_2 \to G_1$ induces an isomorphism $\bar{\psi} : \Ho^2(G_1,\mathbb{C}^{\times}) \to
	\Ho^2(G_2,\mathbb{C}^{\times})$
	if $\bar{\psi}$ is defined by $\bar{\psi}([\alpha]) = [\beta]$ such that $\beta(g,h)= \alpha(\psi(g),\psi(h))$
	for $g,h \in G_2$.
	\begin{theorem} \label{tau(G_1) = tau(G_2)}
		Let $G_1$ and $G_2$ be two non-capable groups and $A_j \subseteq Z^*(G_j) \cap G_j'$ for $j= 1,2$. Suppose there exist isomorphisms $\eta : A_2 \to A_1$ and $\gamma : G_2/A_2 \to G_1/A_1$ such that Diagram~1 is commutative for the induced isomorphisms
		$\bar{\eta}$ and $\bar{\gamma}$. Then, $\tau(G_1) = \tau(G_2)$. 
		\begin{figure}[H]
			\[
			\xymatrix{
				1 \ar[r] & \mathrm{Hom}(A_{1}, \mathbb C^\times) \ar[d]^{\bar{\eta}} \ar[r] ^{\mathrm{tra}} & \Ho^2(G_1/A_{1}, \mathbb C^\times)  \ar[d]^{\bar{\gamma}} \\
				1 \ar[r] & \mathrm{Hom}(A_{2}, \mathbb C^\times) \ar[r] ^{\mathrm{tra}} & \Ho^2(G_2/A_{2}, \mathbb C^\times)\\
			}
			\] \caption{Diagram 1}
		\end{figure}
	\end{theorem}
	
	\begin{proof}
		By Theorem \ref{epicenter}, inf : $\mathrm{H}^2(G_j/A_j, \mathbb{C}^{\times}) \rightarrow\mathrm{H}^2(G_j,\mathbb{C}^{\times})$ is surjective for $j=1, 2$. Let $\rho_i \in \irr^{\alpha}(G_1)$ such that $\rho= \oplus_{i=1}^l \rho_i$ be a faithful $\alpha$-representation of $G_1$ of degree $\tau(G_1)$. By Theorem \ref{direct_sum}, there exists $\tilde{\rho_i} \in \irr^{\beta\tra(\chi_i)}(G_1/A_1)$ such that $\phi(\rho_i)= \tilde{\rho_i} $. Let $\psi_i:= \bar{\eta} (\chi_i)= \chi_i \circ \eta \in \Hom(A_2, \C^\times)$.  
		Consider $\beta'\in Z^2( G_2/A_2, \C^\times)$ such that $\beta'(gA_2,hA_2)= \beta(\gamma(gA_2), \gamma(hA_2))$ for $g,h \in G_2$. Hence, $[\beta']= \bar{\gamma}([\beta])$. Let $\alpha' \in Z^2(G_2,\C^\times)$ defined by $\alpha'(g,h)= \beta'(gA_2,hA_2)$ $\forall$ $g$,$h$ $\in$ $G_2$. Then, by the commutativity of the diagram, we have $\tilde{\rho_i}':= \tilde{\rho_i} \circ \gamma \in \irr^{\beta'\tra(\psi_i)}(G_2/A_2)$, and by Theorem \ref{direct_sum}, there exists $\rho'_i \in \irr^{\alpha'}(G_2)$ such that $\phi(\rho'_i)= \tilde{\rho_i}'$. 
		
		We show that $\oplus_{i=1}^l \rho'_i$ is a faithful $\alpha'$-representation of $G_2$.
		Let $\mu: G_1/A_1 \to G_1$ and $\nu: G_2/A_2 \to G_2$ be sections of $G_1 \twoheadrightarrow G_1/A_1$ and $G_2 \twoheadrightarrow G_2/A_2$, respectively.
		Let $x \in G_2$ such that $(\oplus_{i=1}^l \rho_i')(x)= \lambda I$ for some $\lambda \in \mathbb{C}^{\times}$. Since $x= a\nu(gA_2)$
		for some $gA_2 \in G_2/A_2$, $a\in A_2$, by Remark \ref{remark for bijective correspondence via inflation}(2), $\psi_i(a) \tilde\rho_i'(gA_2)= \lambda I$ $\forall$ $1 \leq i \leq k$. This implies that $(\chi_i \circ \eta(a))(\tilde{\rho_i} \circ \gamma(gA_2))= \lambda I$ and so $\rho_i(\eta(a)\mu(\gamma(gA_2))) = \lambda I$ $\forall i$. Since $\oplus_{i=1}^l \rho_i$ is faithful on $G_1$, $\eta(a) \mu(\gamma(gA_2))=1$. As $\eta(a) \in A_1$, $\mu(\gamma(gA_2)) \in A_1$. Thus, $\gamma(gA_2) =1$. Since $\gamma$ is an isomorphism, $g \in A_2$. Now, $\eta$ is an isomorphism, forcing $a= 1$. Hence, $\oplus_{i=1}^l \rho'_i$ is faithful of degree $\tau(G_1)$. Thus, $\tau(G_2) \leq \tau(G_1)$. By similar arguments, $\tau(G_1) \leq \tau(G_2)$. This completes the proof. 
	\end{proof}
	
	\begin{remark} Let $A_j \subseteq Z(G_j) \cap G_j'$ such that the Diagram 1 given in Theorem \ref{tau(G_1) = tau(G_2)} is commutative for the isomorphisms $\eta$ and $\gamma$. Then, by similar arguments used in the proof of Theorem \ref{tau(G_1) = tau(G_2)}, it can be shown that $\delta(G_1) = \delta(G_2)$.   
	\end{remark}

Note that Theorem \ref{tau(G_1) = tau(G_2)} also provides a method for constructing many additional non-capable groups with the same projective embedding degree, as shown in the following result.

When $G_1$ and $G_2$ satisfy the hypotheses of Theorem \ref{tau(G_1) = tau(G_2)},
we say that the pair $(G_1, G_2)$ satisfies the hypotheses.
\begin{corollary} If two pairs $(G_1, G_2)$ and $(H_1, H_2)$ satisfy the hypotheses of Theorem \ref{tau(G_1) = tau(G_2)}. Then
	
	(i) $\tau(G_1 \times H_1) = \tau(G_2 \times H_2)$.
	
	(ii) For any finite abelian group $A$, $\tau(G_1 \times A) = \tau(G_2 \times A)$. 
\end{corollary} 

\begin{proof}
	(i) By \cite[Proof of Corollary 1.5]{sabnam}, $(G_1 \times H_1, G_2 \times H_2)$ satisfies the hypotheses of Theorem \ref{tau(G_1) = tau(G_2)}, the result follows. \\ 
	(ii) Since $(G_1 \times A, G_2 \times A)$ satisfies the hypotheses of Theorem \ref{tau(G_1) = tau(G_2)}, the result follows.  
\end{proof}

\begin{example} The groups $\Phi_5(311)$, $\Phi_5(2211)a$, $\Phi_5(2211)b, \Phi_5(21^4)c$, listed from \cite{james} have same projective embedding degree, follows from \cite[Proof of Theorem 3.4]{sabnam} and Theorem \ref{tau(G_1) = tau(G_2)}. 
\end{example}



\subsection{Finite abelian groups} \label{abelian}
In this section, we determine $\tau(G)$ for finite abelian groups. It is easy to check that $\tau(G) = \tau_{irr}(G) = 1$ if and only if $G= 1$. 
Recall that the rank of a group $G$ is defined as the minimal number of generators of $G$. Note that if $G$ is a finite abelian group of rank $n$, then $\delta(G)= n$, follows from \cite[Lemma 3.4]{moreto}.    

\begin{lemma} \label{1dim direct sum in abelian group}
	Let $G$ be a finite abelian $p$-group of rank $n$ with $|G|>1$. Suppose $\rho= \oplus_{i=1}^r \rho_i$, where $\rho_i \in \mathrm{Irr}(G)$. If $1 \leq r \leq n$, then $\rho$ is not a faithful projective representation of $G$.   
\end{lemma}

\begin{proof}
	The proof is straightforward for $n=1$. So, we assume $n>1$. First, we consider $G=(\mathbb Z/p\mathbb Z)^n$. Let $e_{k}= \underbrace{(0,0, \ldots, 1, \ldots, 0)}_\text{$1$ in the $k{\text {th}}$ position} \in G$ for $1 \leq k \leq n$.
	
	We assume that $r=n$.  
	Suppose $\rho$ is a faithful projective representation of $G$. Then, for each $k$, there exists $j_{l k} \in \{0,1, \ldots, p-1\}$ ($1\leq l \leq n $) such that all $j_{l k}$'s are not equal and $$\rho(e_{k})=
	\left[\begin{array}{ccc}e^{\frac{2 \pi ij_{1k}}{p}} &  & \\
		& \ddots & \\
		& & e^{\frac{2 \pi ij_{nk}}{p}}\end{array}\right]$$
	is a diagonal matrix. Hence, $$\rho\left(m_{1}, m_{2}, \ldots, m_{n}\right)
	=\left[\begin{array}{ccc}e^{\frac{2 \pi i \sum\limits_{k=1}^{n} m_{k} j_{1 k}}{p}} &  & \\
		& \ddots & \\
		& & e^{\frac{2 \pi i \sum \limits_{k=1}^{n} m_{k} j_{n k}}{p}}\end{array}\right].$$                                            
	Since $\rho$ is also a faithful ordinary representation, $\rho\left(m_{1}, m_{2} \ldots, m_{n}\right)= I$ implies that $\left(m_{1}, \ldots, m_{n}\right)= (0,0, \ldots, 0)$, which says
	$$
	\left[\begin{array}{ccc}
		j_{11} & \cdots & j_{1 n} \\
		\vdots & &\vdots \\
		j_{n 1} & \cdots & j_{n n}
	\end{array}\right]\left[\begin{array}{c}
		x_{1} \\
		\vdots \\
		{x}_{n}
	\end{array}\right]= \left[\begin{array}{c}
		0 \\
		\vdots \\
		0
	\end{array}\right]
	$$
	has only a trivial solution in $\Z/p\Z$. Then, the matrix $A=(j_{lk})_{n\times n}$ is invertible,
	and for $b=[1,1,\cdots, 1]^{\top}$, $A x=b$ has a unique solution $c=A^{-1} b,$ where $c \neq 0$.
	Let $c^{\top}= \left[c_{1}, c_{2}, \ldots ., c_{n}\right]$, 
	$c_{i}$ $\in \{0,1, \ldots, p-1\}$.
	Then, $\sum\limits_{k=1}^{n} j_{i k} c_{k}= 1$ for $1\leq i \leq n$. Therefore, 
	\[\rho\left(c_{1}, c_{2}, \ldots, c_{n}\right)=
	\left[\begin{array}{ccc}
		e\frac{2 \pi i \sum\limits_{k=1}^n c_k j_{1k}}{p} & & \\
		& \ddots & \\
		&  & e \frac{2 \pi i \sum\limits_{k=1}^n c_k j_{nk}}{p} 
	\end{array}\right]
	= e^{\frac{2 \pi i}{p}} I_{n \times n},
	\] which contradicts our assumption. Hence, $\rho$ cannot be faithful. 
	
	If $r<n$, let $\psi= \rho \oplus (\oplus_{i=1}^{n-r} 1_G$). If $\rho$ is faithful as projective, then $\psi$ is also faithful as projective, which gives a contradiction to Case 1. Hence, the result holds if $G$ is elementary abelian. 
	Now, let $G$ be any abelian group of rank $n$ and $H= (\Z/p\Z)^n \subseteq G$. If $\rho$ is faithful on $G$, then it is also faithful on $H$, which is not possible. This completes the proof.
	
\end{proof}

\begin{lemma} \label{abelian group of small rank}
	Let $G$ be a finite abelian $p$-group of rank $n$ such that $n \le p-1$ and $|G|>1$. Then, $\tau(G)= n+1$.    
\end{lemma} 

\begin{proof}
	Since $\delta(G)=n$, by Lemma \ref{relation}, we obtain that $\tau(G) \le n+1$. 
	Let $\rho$ be a faithful $\alpha$-representation of $G$ of degree $k$, where $k =\tau(G)$. 
	If $k< n+1$, then $\rho = \oplus_{i=1}^k\rho_i$, where each $\rho_i \in \irr^\alpha(G)$ has degree $1$. Since $\rho_is$ are one-dimensional, $\alpha$ is a coboundary. Therefore, we can assume that $\rho$ is an ordinary representation. 
	However, by Lemma \ref{1dim direct sum in abelian group}, $k> n$. Hence, the result follows.
\end{proof}
\begin{theorem} \label{elementary abelian p-group}
	Let $G \cong (\mathbb Z/p\Z)^{n}$. Then, the following holds. \\
	(i) $\tau(G) = n+1$ if $G \ncong 1, (\Z/2\Z)^2$ and $(\Z/2\Z)^4$.  \\
	(ii) $\tau(G)= n$ if $G \cong 1, (\Z/2\Z)^2$ or $ (\Z/2\Z)^4$.
\end{theorem}

\begin{proof}
	By Lemma \ref{abelian group of small rank}, it is enough to consider $n \geq p$. Let $\rho$ be a faithful $\alpha$-representation of $G$ of degree $\tau(G)$ such that $\rho = \oplus_{i=1}^m\rho_i$, where $\rho_i \in \mathrm{Irr}^\alpha(G)$. Then, there is an $s \geq 0$ such that $\operatorname{deg}(\rho_i) = p^s$ for all $i$, follows from Lemma \ref{properties of abelian groups}. Since $\delta(G)= n$, by Lemma \ref{relation}, $\tau(G) \leq n+1$. First, assume that $\tau(G) < n+1$. \\
	
	$(i)$ Let $s \geq 1$. Then $n = kp^s+r$ for some $ k \geq 0$ and $0 \leq r < p^s$. If $k=0$, then $n=r <p^s \leq \operatorname{deg}(\rho)= \tau(G)$. So, we assume $k \ge 1$. If $m \geq (k+1)$, then $\operatorname{deg}(\rho) > n$. Hence, we consider $m \leq k$. 
	By Lemma \ref{properties of abelian groups}, $\alpha|_{G_0\times G_0}$ is a coboundary, so $\rho|_{G_0}$ is projectively equivalent to an ordinary representation $\rho'$ of $G_0$ and we may assume that $\rho' = \oplus_{i=1}^m \deg(\rho_i)\rho_i'$, where $\rho_i' \in \mathrm{Hom}(G_0, \mathbb C^\times)$. 
	Since $\rho|_{G_0}$ is faithful, $\rho'$ is also faithful by Lemma \ref{projectively equivalent}. Therefore, $\oplus_{i=1}^m \rho_i'$ is faithful on $G_0$. Since $G_0 \cong (\mathbb Z/p\Z)^{n-2s}$ (by Lemma \ref{properties of abelian groups}), it follows from Lemma \ref{1dim direct sum in abelian group} that rank$(G_0)=n-2s < m \leq k$ which is true only for the following cases: 
	\begin{enumerate}
		\item $p=2$, $s=2$, $k=1$, $r=0$, i.e., $G \cong (\Z/2\Z)^4$. 
		\item $p=2$, $s=1$, $k=1$, $r=0$, i.e., $G \cong (\Z/2\Z)^2$.  
	\end{enumerate}
	Thus, we conclude that for $G \ncong (\Z/2\Z)^2$ and $(\Z/2\Z)^4$, $\rho$ cannot be faithful when $s \geq 1$ and $\operatorname{deg}(\rho) < n+1$. \\
	For $s=0$, $\rho$ is a direct sum of one-dimensional representations, so by Lemma \ref{1dim direct sum in abelian group}, $\operatorname{deg}(\rho) \geq$ $n+1$. Hence, $\tg = n + 1$ when $G \ncong 1, (\Z/2\Z)^2, (\Z/2\Z)^4$. \\
	
	$(ii)$ If $G \cong (\Z/2\Z)^2$, by Theorem \ref{projective irreducible embedding degree of finite abelian groups}, $\tg=2$. 
	
	If $G \cong (\Z/2\Z)^4$, taking a subgroup $H =(\Z/2\Z)^3$, we have $\tau(H) = 4$ $\leq $ $\tg$. Now, by Theorem \ref{projective irreducible embedding degree of finite abelian groups}, $\tg = 4$.
	
\end{proof}

\begin{theorem}\label{projective embedding degree of finite abelian groups}
	Let $G$ be a finite abelian group of rank $n$ which is not elementary abelian. Then $\tau(G) = n + 1$.
	
\end{theorem}

\begin{proof}
	Let $G$ $\cong$ $\mathbb{Z}/{m_1}\Z \times \mathbb{Z}/{m_2}\Z \times \cdots \times \mathbb{Z}/{m_n}\Z$, where $m_{i+1} \mid m_i$ for $1 \leq i \leq n-1$. Let $p$ be the smallest prime dividing $m_n$. 
	
	Suppose $p = n= 2$. Then $o(G) > 4$ and consider $H_1 = \Z/2q_1\Z \times \Z/2\Z \subseteq G$ for prime $q_1$ such that $2q_1|m_1$. Since $\delta(H_1)= 2$, by Lemma \ref{relation}, $\tau(H_1) \leq 3$. Let $\psi$ be a faithful projective representation of $H_1$ of degree $\tau(H_1)$. Since $H_1$ is not of symmetric type, $\psi$ is not irreducible. Hence, using Lemma \ref{1dim direct sum in abelian group}, we have $\tau(H_1) >2$, and thus $3 =\tau(H_1)$ $\leq$ $\tau(G)$. Since $\delta(G)= 2$, Lemma \ref{relation} yields that $\tg = 3= n+1$. 
	
	If $p =2, n= 4$, then $o(G) > 16$ and consider $H = \Z/2q_2\Z \times$ $(\Z/2\Z)^3$, where $q_2$ is a prime such that $2q_2|m_1$. Since $\delta(H)= 4$, by Lemma \ref{relation},  $\tau(H) \leq 5$. Let $\psi$ be a faithful $\alpha$-representation of $H$ of degree $\tau(H)$. 
	Since $H$ is not of symmetric type, $\psi$ is not irreducible. Using Lemma \ref{1dim direct sum in abelian group}, if deg$(\psi) <5$, we can assume $\psi= \oplus_{i=1}^2\psi_i$ such that $\operatorname{deg}(\psi_i)=2$. Let $H_0$ be the set of $\alpha$-regular elements of $H$. Then, $H_0= \Z/2q_2\Z \times$ $(\Z/2\Z)^2$ or $H_0= \Z/q_2\Z \times$ $(\Z/2\Z)^3$. Since rank$(H_0)>2$, proceeding along the same lines as proof of Theorem \ref{elementary abelian p-group}$(i)$, we conclude that $\psi$ cannot be faithful. Hence, $\tau(H) = 5 \leq$ $\tau(G)$. Since $\delta(G)= 4$, $\tg = 5= n+1$ in view of Lemma \ref{relation}. 
	
	In all other cases, taking $H = (\Z/p\Z)^n$ $\subseteq$ $G$, by Theorem \ref{elementary abelian p-group}, we have $\tau(G)$ $\geq \tau(H) = n + 1$. Since $\delta(G)= n$, Lemma \ref{relation} yields that $\tg = n + 1$. 
	
\end{proof}

\subsection{Extra-special $p$-groups and Heisenberg groups} \label{EH} 
In this section, we study $\tau(G)$ and $\tau_{irr}(G)$ for extra-special $p$-groups and Heisenberg groups. We observe that Theorem \ref{direct_sum} plays a crucial role in these computations.


We first recall the following part from \cite[Page 772]{HIGGS}, which will be used in the proof of the results presented in this section. 
Let $\alpha$ be a cocycle of an abelian group $G$. Define a map $\alpha' \colon G \times G \to \mathbb{C}^{\times}$ by $\alpha'(x,y) ={\alpha(x,y)}{\alpha(y,x)^{-1}}$ $\forall$ $x,y$ $\in G$. Then, $\alpha'$ is independent of the choice of cocycle representative in $[\alpha]$. Suppose that
$$G \cong \langle x_1 \rangle \times \cdots \times \langle x_n \rangle \cong (\Z/p^k\Z)^n,$$ where each $x_i$ has order $p^k$. Let $\omega$ be a primitive $p^k$th root of unity. Then, for all $i,j \in \{1,\ldots,n\}$, we can write
\[
\alpha'(x_i,x_j)=\omega^{c(i,j)}.
\]
We define the alternating $n \times n$ matrix $C(\alpha)= [c(i,j)]$ over $\Z/p^k\Z$ representing $[\alpha]$. Then, the map $[\alpha] \mapsto C(\alpha)$ is an isomorphism from $\mathrm{H^2}(G, \mathbb{C}^{\times})$ onto the group $R(p^k,n)$ of all alternating $n \times n$ matrices over $\Z/p^k\Z$, under matrix addition.
Let $x= x_1^{b_1}\dots x_n^{b_n}\in G$. Then, $x$ is an $\alpha$-regular element if and only if $C(\alpha)b= 0$, where $b^{\top}= [b_1,...,b_n]$.
Furthermore, if $k=1$, then $\operatorname{rank}\big(C(\alpha)\big) = 2r$,
where $p^r$ is the common degree of all irreducible $\alpha$-representations of $G$. 

\begin{remark} \label{ker(C(alpha))} Note that $|\operatorname{ker}\big(C(\alpha)\big)| = |G_0|$.   
\end{remark}

In the following result, we determine $\tau_{irr}(G)$ and $\tau(G)$ for extra-special $p$-groups of order $p^{2n+1}$, with $n \geq 1$.
Since groups of order $p^3$ will be discussed in Section \ref{groupsoforderp5}, we consider $n> 1$ here.

\begin{theorem} \label{extraspecial $p$-group}
	Let $G$ be an extra-special $p$-group of order $p^{2n+1}$, where $n>1$. Then, $\tau_{irr}(G)$ does not exist and  
	\begin{equation}\nonumber
		\tau(G) =
		\begin{cases}
			2p^{n/2}, & \text{if n is even }  \\
			p^{(n+1)/2}+p^{(n-1)/2},  & \text{if n is odd } 
		\end{cases}
	\end{equation}
\end{theorem}

\begin{proof}
	By \cite[Chapter 4, Theorem 7.1]{karpilovsky}, extra-special $p$-groups are central products of their subgroups of order $p^3$. Hence, by Lemma \ref{central product} and Corollary \ref{non-capable groups}, $\tau_{irr}(G)$ does not exist. 
	
	From the proof of Lemma \ref{central product}, it follows that $\mathrm{inf}: \mathrm{H}^2(G/Z(G), \C^{\times}) \to \mathrm{H}^2(G, \C^{\times})$ is surjective.  
	Let $Z(G) = \langle z \rangle \cong \Z/p\Z$ and $G/Z(G) = \prod_{k=1}^{2n}\langle x_i \rangle \cong (\Z/p\Z)^{2n}$. 
	Recall the transgression map 
	$$\mathrm{tra} : \mathrm{Hom}(Z(G), \mathbb{C}^{\times}) \to \mathrm{H}^2(G/Z(G), \mathbb{C}^{\times})$$ from \cite[Proof of Theorem 1.7]{sabnam}, which is
	defined as follows. For $X = \prod_{k=1}^{2n} x_k^{i_k}$,
	$Y = \prod_{k=1}^{2n} x_k^{j_k} $ in $G/Z(G)$, $$\mathrm{tra}(\chi)(X, Y) = \chi(z)^{-\sum_{k=1}^n j_k i_{n+k}}$$ for  $\chi \in \mathrm{Hom}(Z(G), \mathbb{C}^{\times})$. Let $\omega$ be a primitive $p$th root of unity. Then $\chi(z)= \omega^r$ for some $r\in \{0,1,\dots,p-1\}$.
	Hence, $C(\mathrm{tra}(\chi))= 
	\begin{bmatrix}
		0 & rI_{n \times n} \\
		-rI_{n \times n} & 0
	\end{bmatrix}$.
	It is easy to see that $\text{rank}\big(C(\mathrm{tra}(\chi)\big) = 2n$ for $r\neq 0$. 
	
	Let $\tau(G)= m$ and $\rho_i \in \mathrm{Irr}^\alpha(G)$ such that $\rho = \oplus_{i=1}^l\rho_i$ be a faithful $\alpha$-representation of $G$ of degree $m$. By Theorem \ref{direct_sum}, taking $A= Z(G)$, we have $\tilde\rho_i \in \mathrm{Irr}^{\beta\mathrm{tra}(\chi_i)}\big(G/Z(G)\big)$, where $\chi_i \in  \Hom(Z(G), \mathbb{C}^{\times})$, such that $\oplus_{i=1}^l\chi_i$ is faithful on $Z(G)$. 
	Furthermore, we can assume that $\chi_i \neq \chi_j$ for $i \neq j$.
	Since $\chi_1 \oplus \chi_2$ is faithful on $Z(G)$, we have $l=2$. If $\beta_1= \beta\mathrm{tra}(\chi_1)$, then $\beta\mathrm{tra}(\chi_2)= \beta_1\mathrm{tra}(\chi_3)$, where $\chi_3= \chi_1^{-1}\chi_2 \in \mathrm{Hom}(Z(G), \mathbb{C}^{\times})$. Hence, to compute $\tau(G)$, it is enough to find 
	\[\operatorname{min}\Biggl\{
	\begin{array}{cl} \mathrm{deg}(\tilde{\rho}) \mid & \tilde{\rho}= \tilde{\rho}_1 \oplus \tilde{\rho}_2,\ \tilde{\rho}_1 \in \irr^{\beta_1}(G/Z(G)) \text{ and } \tilde{\rho}_2 \in \mathrm{Irr}^{\beta_1\tra(\chi)}(G/Z(G)),\\
		& \chi\in\operatorname{Hom}(Z(G),\mathbb{C}^{\times}), \chi \neq 1,[\beta_1] \in \mathrm{H}^2\big(G/Z(G), \C^{\times}\big)
	\end{array}
	\Bigg\}\]
	which is equivalent to consider
	\[\operatorname{min}\Biggl\{
	\begin{array}{cl}  \text{rank}\big(C(\beta_1)\big)+\text{rank}\Big(C\big(\beta_1\mathrm{tra}(\chi)\big)\Big)\mid & [\beta_1] \in \mathrm{H}^2(G/Z(G), \C^{\times}),\\
		& \chi \in \operatorname{Hom}(Z(G),\mathbb{C}^{\times}), \chi \neq 1 \end{array}
	\Bigg\}.\] 
	Let $\operatorname{deg}(\rho_1)=\operatorname{deg}(\tilde{\rho}_1)= p^{k_1}$ and $\operatorname{deg}(\rho_2)=\operatorname{deg}(\tilde{\rho}_2)=p^{k_2}$.
	Now, $\text{rank} \big(C(\beta_1)\big) +\text{rank} \Big(C\big(\beta_1\mathrm{tra}(\chi)\big)\Big)=$
	$\text{rank}\big(C(\beta_1)\big)+\text{rank}\Big(C\big(\beta_1\big)+C\big(\mathrm{tra}(\chi)\big)\Big)$
	$\geq \text{rank}\Big(C\big(\mathrm{tra}(\chi)\big)\Big)=2n$. Therefore, $k_1+k_2 \geq n$. It is easy to check that 
	$$\operatorname{min}\{p^{k_1}+p^{k_2}\mid k_1+k_2 \geq n\} = p^{\lfloor{n/2}\rfloor} +p^{\lceil{n/2}\rceil}.$$  
	Taking $C(\beta_1)=
	\begin{bmatrix}
		0 & 
		\begin{bmatrix}
			-L & 0 \\[4pt]
			0 & 0
		\end{bmatrix}
		\\[10pt]
		\begin{bmatrix}
			L & 0 \\[4pt]
			0 & 0
		\end{bmatrix}
		& 0
	\end{bmatrix},
	\quad \text{where} \quad L = I_{\left\lfloor \frac{n}{2} \right\rfloor \times \left\lfloor \frac{n}{2} \right\rfloor}$ and $\chi \in  \operatorname{Hom}(Z(G),\mathbb{C}^{\times})$ given by $\chi(z^i)=\omega^{i}$, we get $\operatorname{deg}(\rho) = p^{\frac{\mathrm{rank}(C(\beta_1))}{2}}+p^{\frac{\mathrm{rank}(C(\beta_1 tra(\chi)))}{2}}
	=p^{\lfloor{n/2}\rfloor} + p^{\lceil{n/2}\rceil}$. Hence, the result follows.
	
\end{proof}
Note that while writing a presentation of a group, we omit relations of the form $[x, y] = 1$ for generators $x, y \in G$.  
Our next aim is to describe $\tau(G)$ and $\tau_{irr}(G)$ for the Heisenberg groups $G = H_{2n+1}(\Z/p^k\Z)$, which admit the presentation: 
$$H_{2n+1}(\Z/p^k\Z)= \langle x_i, y_i, z, 1\leq i \leq n \mid [x_i, y_i] = z, x_i^{p^k} = y_i^{p^k} = z^{p^k} = 1 \rangle.$$
\begin{theorem} \label{Heisenberg group}
	Let $p$ be an odd prime. 
	\begin{enumerate}
		\item Suppose $G$ is a $p$-group such that $G' \subseteq Z(G)$ and $Z(G)$ is cyclic. If $\tau_{irr}(G)$ exists, then $G \cong H_{3}(\Z/p^k\Z)$ and in that case, $\tau_{irr}(G)= \tau(G)= p^k$.  
		
		\item If $G= H_{2n+1}(\Z/p^k\Z)$ with $n > 1$, then $\tau_{irr}(G)$ does not exist and $\tau(G) \leq p^{k\lfloor{n/2}\rfloor} + p^{k\lceil{n/2}\rceil}$.
	\end{enumerate}
	
\end{theorem}

\begin{proof}(1)  
	From the proof of \cite[Proposition 7.6]{NG}, it follows that if $G' \subseteq Z(G)$, $Z(G)$ is cyclic and $\tau_{irr}(G)$ exists, then $G \cong H_{3}(\Z/p^k\Z)$.
	Now, by \cite[Theorem 1.2]{hatui}, consider the representation group of $G= H_{3}(\Z/p^k\Z)$ given by
	$$G^* = \langle \x, \y, \z, z_1, z_2 \mid [\x, \y] = \z, [\z, \x] = z_1, [\z, \y] = z_2, 
	\x^{p^k} = \y^{p^k} = \z^{p^k} = 1 \rangle. $$
	Then, $G^*/\langle z_1, z_2 \rangle \cong G$. Let $A= \langle z_1, z_2 \rangle$.
	By \cite[Lemma 2.2]{hatui}, any cocycle of $G$ is cohomologous to a cocycle $\alpha$ of the form 
	\[
	\alpha\big(z^{m_1}y^{n_1}x^{q_1}, z^{m_2}y^{n_2}x^{q_2}\big) = \lambda^{(m_2q_1 + n_2 \frac{q_1(q_1-1)}{2})} \mu^{(n_1m_2+q_1 \frac{n_2(n_2-1)}{2}+q_1n_1n_2)} , 
	\]
	for some $\lambda, \mu \in \mathbb{C}^{\times}$ such that $\lambda^{p^k} = \mu^{p^k} = 1$.  
	It is easy to check that $1$ is the only $\alpha$-regular element of $Z(G)$ if and only if either $\lambda$ or $\mu$ is a primitive $p^k$th root of unity. 
	Thus, by Lemma \ref{faithful projective representation}, $G$ admits a faithful irreducible $\alpha$-representation only for those $\alpha$ having $\lambda$ or $\mu$ a primitive $p^k$th root of unity. Let $S$ be the set consisting of all such $[\alpha]$. 
	Therefore, by Theorem \ref{bijective correspondence} we have
	\begin{eqnarray*}
		\tau_{irr}(G)&=&\operatorname{min}\{\operatorname{deg}(\operatorname{\rho})\mid \rho \in \mathrm{Irr}^\alpha(G) \text{ for } [\alpha] \in S\}  \\
		&=&\operatorname{min}\{\operatorname{deg}(\operatorname{\Gamma}) \mid \Gamma \in \mathrm{Irr}(G^*\mid \chi), \chi\in \Hom(A, \C^\times) \text{ such that } \operatorname{tra}(\chi) \in S\}.
	\end{eqnarray*}
	Define a set $T=\{\chi \in \Hom(A,\C^\times) \mid \chi(z_1) \text{ or } \chi(z_2) \text{ is a primitive } p^k\text{th root of unity} \}$.
	First, we show that $\operatorname{tra}(\chi) \in S$ if and only if $\chi \in T$. If $\chi^{p^m}=1$ for some $m<k$, then $\operatorname{tra}(\chi)^{p^m}=1$.   
	Hence, if $\operatorname{tra}(\chi) \in S$, then $\chi \in T$. Conversely, let $\chi \in T$ and $\nu: G \to G^*$ be a section defined by $\nu(z^{m}y^{n}x^{q}) = \z^{m}\y^{n}\x^{q}$. Then, 
	$$\operatorname{tra}(\chi)(x, z) = \chi(\nu(x)\nu(z)\nu(xz)^{-1})= \chi(\x\z( \z\x)^{-1})= \chi(\z_1^{-1}) \text { and }$$ 
	$$\operatorname{tra}(\chi)(y, z) = \chi(\nu(y)\nu(z)\nu(yz)^{-1})= \chi(\y\z( \z\y)^{-1})= \chi(\z_2^{-1}).$$ 
	Therefore, $\operatorname{tra}(\chi) \in S$.
	Hence, $$\tau_{irr}(G)= \operatorname{min}\{\operatorname{deg}(\operatorname{\Gamma}) \mid \Gamma \in \mathrm{Irr}(G^*\mid \chi), \chi \in T\}.$$
	Consider the abelian normal subgroup $N = \langle \z, z_1, z_2 \rangle$ of $G^*$ of order $p^{3k}$. Let $\chi \in T$. WLOG, we assume that $\chi(z_1) = \chi(z_2)^a$, where $\chi(z_2)$ is a primitive $p^k$th root of unity and $0 \leq a \leq (p^{k}-1)$. 
	Let $\chi_1 \in \text{Irr}(N)$ such that $\chi_1|_{\langle z_1, z_2 \rangle} = \chi$. It is easy to check that $I_{G^*}(\chi_1)= \langle N,\x^{l}\y^{-al}\rangle$ and $G^*/I_{G^*}(\chi_1)$ is cyclic of order $p^k$. It also follows that, for $\Gamma \in \operatorname{Irr}(G^* \mid \chi)$, deg$(\Gamma) = p^k$ by Theorem \ref{Clifford's theory}. Hence, $\tau_{irr}(G)= p^k$. 
	
	
	
	Now we claim that $\tau(G) = p^k$. 
	Since $\tau_{irr}(G)= p^k$, $\tau(G) \leq p^k$. 
	WLOG, we assume that $\chi'(z_1) = \chi'(z_2)^a$, where $\chi'(z_2)$ is a primitive $p^r$th root of unity, $0 <r <k$ and $0 \leq a \leq (p^{r}-1)$.  Let $\chi \in \text{Irr}(N)$ such that $\chi|_{\langle z_1, z_2 \rangle} = \chi'$.
	It is easy to check that $I_{G^*}(\chi) = \langle N,\x^{l}\y^{-al},\y^{p^r} \rangle$. 
	Let $N_1= \langle N,\x^{l}\y^{-al} \rangle$. As $N_1/N$ is cyclic, $\chi$ extends to some $\bar{\chi}\in \Hom(N_1, \C^\times)$ by Theorem \ref{Clifford's theory}. Note that $I_{G^*}(\chi)/N_1 = \langle \y^{p^r}N_1 \rangle$. 
	
	
	First, we study the dimensions and behaviour of $\Gamma \in \operatorname{Irr}(G^* \mid \chi)$ as $\Gamma$ is a lift of $\rho \in \irr^\alpha(G)$ for $\alpha= \tra(\chi')$. Let $n$ be the smallest positive integer such that $r \leq n\leq k$ and $\bchi^{(\y^{p^n})}(\x\y^{-a}) = \bchi(\x\y^{-a})$. 
	Then, we can check that $\bchi^{(\y^{p^n})}(\x^l\y^{-al}) = \bchi(\x^l\y^{-al})$. Since $\bchi^{(\y^{p^n})}(\z) = \bchi(\z z_2^{-p^n})= \bchi(\z)$, $\bchi$ extends to $N_2= \langle N,\x^l\y^{-al}, \y^{p^n} \rangle =\langle N,\x\y^{-a}, \y^{p^n} \rangle$. Let $\sigma: N_2 \to \C^\times$ be such that $\sigma|_{N_1} = \bchi$. 
	Note that $I_{G^*}(\bar{\chi}) =  N_2$. Hence, by Theorem \ref{Clifford's theory}, every $\psi \in \operatorname{Irr}(I_{G^*}(\chi) \mid \bar{\chi})$ is of the form $\psi= \operatorname{Ind}^{I_{G^*}(\chi)}_{N_2}(\sigma)$ and $\Gamma = \operatorname{Ind}^{G^*}_{{I_{G^*}(\chi)}}(\psi)$. Hence, $\operatorname{deg}(\Gamma)= p^n$. Note that $\{\y^i$, $i=1,2,..,p^r\}$ and $\{\y^{ip^r}$, $i=1,2,..,p^{n-r}\}$ are sets of left coset representatives of $I_{{G}^*}(\chi)$ in $G^*$ and of $N_2$ in $I_{{G}^*}(\chi)$, respectively. Then, 
	$$\psi(\z)= \left[\begin{array}{cccc}
		\chi(\z)\chi(z_2)^{-p^r} & & & \text{\huge0} \\
		& \chi(\z)\chi(z_2)^{-2p^r} & & \\
		& & \ddots & \\
		\text{\huge0} & & & \chi(\z)\chi(z_2)^{-p^n} 
	\end{array}\right]_{p^{n-r} \times p^{n-r}}= \chi(\z)I_{p^{n-r} \times p^{n-r}}$$ and  $$\Gamma(\z)= \left[\begin{array}{ccc}
		\psi(\z)\psi(z_2)& & \text{\huge0} \\
		& \ddots & \\
		\text{\huge0} &  & \psi(\z)\psi(z_2^{p^r})
	\end{array}\right]_{p^n \times p^n} = \chi(\z)\left[\begin{array}{ccc}
		\psi(z_2)& & \text{\huge0} \\
		& \ddots & \\
		\text{\huge0} &  & \psi(z_2^{p^r})
	\end{array}\right]_{p^n \times p^n}.$$ 
	
	\bigskip
	
	Now, we show that if $\alpha= \tra(\chi')$ and $\rho$ is an $\alpha$-representation of $G$ with $\deg(\rho)< p^k$, then $\rho$ cannot be faithful.
	Let $\rho = \oplus_{i=1}^m \rho_i$, where $\rho_i \in \mathrm{Irr}^\alpha(G)$ and $\deg(\rho_i)<p^k$. Let $\Gamma_i \in \irr(G^* \mid \chi')$ correspond to $\rho_i$, where the correspondence is given by Theorem \ref{bijective correspondence}. Then, $\Gamma_i = \operatorname{Ind}^{G^*}_{{I_{G^*}(\chi_i)}}(\psi_i)$, where $\chi_i \in \text{Irr}(N)$ 
	such that $\chi_i |_A= \chi'$, and $\psi_i=$ $\operatorname{Ind}^{I_{G^*}(\chi_i)}_{N_2}(\sigma_i)$ for $\sigma_i \in \Hom(N_2, \C^\times)$ such that $\sigma_i|_{N_1} = \chi_i$. Let $r \leq n_i < k$ be the smallest positive integer such that $\chi_i^{(\y^{p^{n_i}})}(\x\y^{-a}) = \chi_i(\x\y^{-a})$. By viewing the above matrices, $\rho_i(z^{p^{k-1}})= \Gamma_i(\z^{p^{k-1}})= \chi_i(\z^{p^{k-1}})$. If none of $\chi_i(z)$ is a primitive $p^k$th root of unity, then $\rho(z^{p^{k-1}})= I$ and $\rho$ is not faithful.
	Hence, for $\rho|_{Z(G)}$ to be faithful, at least one of the $\chi_i(\z)$ has to be a primitive $p^k$th root of unity.
	Now, for any $n \geq r$, we have $\chi_i^{(\y^{p^{n}})}(\x\y^{-a}) = \chi_i(\x\y^{-a}\z^{-p^n}z_1^{-p^n\binom{1}{2}}z_2^{ap^n+\binom{p^k-p^n}{2}})=\chi_i(\x\y^{-a} \z^{-p^n})$.
	Then, $\chi_i^{(\y^{p^{n_i}})}(\x\y^{-a}) \neq \chi_i(\x\y^{-a})$, which is a contradiction. Hence, $\rho|_{Z(G)}$ is not faithful. In view of Theorem \ref{restriction}, $\rho$ cannot be faithful. 
	
	
	Now, let $\chi' \in \text{Irr}(\langle z_1, z_2 \rangle)$ such that $\chi'(z_1) = \chi'(z_2) = 1$. Then, $\tra(\chi')= 1$. Let $\rho\in \irr(G)$ be faithful. Since $\dg= p^k$ (\cite[Theorem 3.1, (8)]{kaur}), $\deg(\rho) \geq p^k$. This completes the proof. \\
	
	
	
	
	
	
	(b) For $n > 1$, the group $H_{2n+1}(\Z/p^k\Z)$ is a central product of $H_{2n-1}(\Z/p^k\Z)$ and $H_{3}(\Z/p^k\Z)$. Therefore, by Lemma \ref{central product} and Corollary \ref{non-capable groups}, $\tau_{irr}(G)$ does not exist, and the proof of the Lemma \ref{central product} also shows that $\mathrm{inf}: \mathrm{H}^2(G/Z(G), \C^{\times}) \to \mathrm{H}^2(G, \C^{\times})$ is surjective. Since $Z(G) = \langle z \rangle \cong \Z/p^k\Z$,  let $G/Z(G) = \prod_{k=1}^{2n}\langle a_i \rangle \cong (\Z/p^k\Z)^{2n}$. Using arguments similar to those in \cite[Proof of Theorem 1.7]{sabnam}, one can check that the transgression map $\mathrm{tra} : \mathrm{Hom}(Z(G), \mathbb{C}^{\times}) \to \mathrm{H}^2(G/Z(G), \mathbb{C}^{\times})$ is defined as follows. For $X = \prod_{k=1}^{2n} a_k^{i_k}$, $Y = \prod_{k=1}^{2n} a_k^{j_k} $ in $G/Z(G)$, $$\mathrm{tra}(\chi)(X, Y) = \chi(z)^{-\sum_{k=1}^n j_k i_{n+k}}$$ for $\chi \in \mathrm{Hom}(Z(G), \mathbb{C}^{\times})$. Let $\omega$ be a primitive $p^k$th root of unity. Then, $\chi(z)= \omega^r$ for some $r\in \{0,1,\dots,p^k-1\}$.
	Hence, $C\big(\mathrm{tra}(\chi)\big)= 
	\begin{bmatrix}
		0 & rI_{n \times n} \\
		-rI_{n \times n} & 0
	\end{bmatrix}$.
	
	Let $\tau(G)= m$ and $\rho$ be a faithful $\alpha$-representation of $G$ of degree $m$. Then, $\rho= \oplus_{i=1}^l\rho_i$, where $\rho_i \in \irr^{\alpha}(G)$. By Theorem \ref{direct_sum} (taking $A= Z(G)$), $\tilde\rho_i \in \mathrm{Irr}^{\beta\mathrm{tra}(\chi_i)}\big(G/Z(G)\big)$, where $\chi_i \in  \Hom(Z(G), \mathbb{C}^{\times})$ such that $\oplus_{i=1}^l\chi_i$ is faithful on $Z(G)$. 
	Furthermore, we can assume that $\chi_i \neq \chi_j$ for $i \neq j$. If $\chi_i(z)$ is not a primitive $p^k$th root of unity for any $i$, then $\oplus_{i=1}^l\chi_i$ cannot be faithful on $Z(G)$. Hence, WLOG, let $\chi_1(z)$ be a primitive $p^k$th root of unity. Since $\chi_1 \oplus \chi_2$ is faithful on $Z(G)$, we have $l=2$. 
	Hence, it is enough to find 
	
	\[\operatorname{min}\Biggl\{
	\begin{array}{cl}  \mathrm{deg}(\tilde{\rho}) \mid & \tilde{\rho}=\tilde{\rho}_1 \oplus \tilde{\rho}_2,\ \tilde{\rho}_1 \in \irr^{\beta\tra(\chi_1)}(G/Z(G)) \text{ and }\tilde{\rho}_2 \in \mathrm{Irr}^{\beta \tra(\chi_2)}(G/Z(G)), \\
		& \chi_i \in  \operatorname{Hom}\big(Z(G),\mathbb{C}^{\times}\big), i=1,2 \text{ and } \chi_1(z) \text{ is a primitive } p^k\text{th root of unity} 
	\end{array}
	\Bigg\}\] 

	Let $C(\beta)=
	\begin{bmatrix}
		0 & 
		\begin{bmatrix}
			-L & 0 \\[4pt]
			0 & 0
		\end{bmatrix}
		\\[10pt]
		\begin{bmatrix}
			L & 0 \\[4pt]
			0 & 0
		\end{bmatrix}
		& 0
	\end{bmatrix}
	\quad \text{, where} \quad L = I_{\left\lfloor \frac{n}{2} \right\rfloor \times \left\lfloor \frac{n}{2} \right\rfloor}$. Let $\chi \in  \operatorname{Hom}(Z(G),\mathbb{C}^{\times})$ be defined by $\chi(z^i)= \omega^{i}$. Let $\tilde{\rho_1}$ and $\tilde{\rho_2}$ be irreducible $\beta$ and $\beta\tra(\chi)$-representations of $G/Z(G)$. Then, by Remark \ref{ker(C(alpha))} and Lemma \ref{properties of abelian groups}, $\deg(\tilde{\rho_1}) = p^{k{\lfloor{n/2}\rfloor}}$ and $\deg(\tilde{\rho_2})= p^{k{\lceil{n/2}\rceil}}$. 
	Let $\alpha(g,h) = \beta(gZ(G), hZ(G))$. Then, $\alpha \in Z^2(G, \C^\times)$. Let $\rho_i \in \irr^\alpha(G)$ correspond to $\tilde{\rho_i}$, $i=1,2$, where the correspondence is given by Remark \ref{remark for bijective correspondence via inflation}. Since $1|_{Z(G)} \oplus \chi$ is faithful on $Z(G)$, Theorem \ref{direct_sum} implies that $\rho_1 \oplus \rho_2$ is faithful on $G$. Hence, $\tg \leq p^{k{\lfloor{n/2}\rfloor}} + p^{k{\lceil{n/2}\rceil}}$. 
\end{proof}
\section{Examples} \label{additional} 


In this section, we discuss several examples illustrating $\tau(G)$ and $\tau_{irr}(G)$ for certain classes of groups $G$, particularly metacyclic groups, finite simple groups and symmetric groups. 
We begin by recalling the following facts, which will be useful in discussing the examples.
\begin{enumerate}
	\item  If $G= D_{2n}$ or $Q_{4n}$, then $\delta(G)= 2$ by \cite[Example 2.2]{kaur}, where $D_{2n}$ and $Q_{4n}$ denote the Dihedral group of order $2n$ and the Quaternion group of order $4n$, respectively.
	\item If $G= S_n, n \geq 5$ or $A_n, n \geq 6$, then $\delta(G) = \delta_{irr}(G)= n-1$ (\cite[Example 2.4, 2.5, 2.6]{kaur}).
	\item  For $n \geq 4$, $M(S_n) \cong \Z/2\Z$ by \cite[Chapter 4, Theorem 8.3(i)]{karpilovsky}. 
	\item For $[\alpha] \neq 1$, every $\alpha$-representation has even degree by \cite[Chapter 5, Theorem 5.3]{karpilovsky3}. 
	
	\item A metacyclic $p$-group $G$, $p \ge 3$, has a faithful irreducible projective representation if and only if $G$ is of symmetric type, i.e., $G$ can be written as a semidirect product of two isomorphic cyclic subgroups  (see \cite [Theorem 5.2]{NG}).
\end{enumerate}

\textbf{Example 5.1.} \label{tau(G) = 2}
By \cite[Section 1.1]{Dolgachev}, the finite subgroups of $\mathrm{PGL}_2(\C)$ are isomorphic to $\Z/n\Z$, $D_{2n}$, $A_4$, $A_5$ or $S_4$. Hence, $\tg = 2$ if and only if $G$ is one of these groups. 
From Lemma \ref{one dim direct sum in non-abelian group}(1), it follows that if $G$ is non-abelian and $\tg = 2$, then $\ti= 2$. Also, if $\ti= 2$, then $\tg= 2$ and so $G$ is one of the groups $\Z/n\Z$, $D_{2n}$, $A_4$, $A_5$ or $S_4$. Since $\Z/n\Z$ is not of symmetric type, $\tau_{irr}(\Z/n\Z)$ does not exist. Hence, $\ti =2$ if and only if $G$ is one of the groups $D_{2n}$, $A_4$, $A_5$ or $S_4$. \\

\textbf{Example 5.2.} \label{tau(G) = 3}
Let $G$ be a group with $\dg= 2$ such that $G \ncong D_{2n}$. By Lemma \ref{relation}, $\tg \leq 3$. From the character table of $S_4$(\cite[18.1]{liebeck}) and by the proof of Lemma \ref{one dim direct sum in non-abelian group}(1), it follows that $\delta(S_4)=3$. Now,  it follows from Example 5.1 and \cite[Examples 2.2, 2.3 and 2.5]{kaur} that $\tg= 3$.  \\

\textbf{Example 5.3.} 
Let $G$ be a metacyclic $p$-group ($p \ge 3$) of symmetric type. Then, $\tau_{irr}(G)= \sqrt{|G|}$ in view of \cite[Chapter 10, Proof of Theorem 4.16]{karpilovsky}.  \\

\textbf{Example 5.4.}  \label{Q4n}
By Example 5.2, $\tau(Q_{4n})=3$. Moreover, \cite[Chapter 10, Theorem 4.14]{karpilovsky} implies that $\tau_{irr}(Q_{4n})$ does not exist.     \\

\textbf{Example 5.5.} \label{simple group}
Let $G$ be a non-abelian finite simple group.
If $\rho$ is a projective representation of $G$ which is not a direct sum of trivial $1$-dimensional representations, then $\rho$ is faithful. Hence, $\tau(G) = \tau_{irr}(G)$ = minimal degree of an irreducible projective representation greater than $1$.   \\

\textbf{Example 5.6.} \label{tau(G) leq delta(G)}
For $G=S_n$ ($n \geq 3$) or $A_n$ ($n\geq 4$), it follows from Lemma \ref{one dim direct sum in non-abelian group}$(3)$ that $\tau(G) \leq \delta(G)$. \\

\textbf{Example 5.7.} \label{S5}
Consider $G = S_5$. Let $\tau(G)= m$ and $\rho$ be a faithful projective representation of degree $m$. If $\rho$ is an ordinary representation, then $\deg(\rho) \geq \delta(G)= 4$. From the spin character table of $G$ (see \cite[Chapter 5, Theorem 6.8]{karpilovsky3}), it follows that $\deg(\rho) \geq 4$. Hence, $\tau(G) = 4$, in view of Example 5.6. By Lemma \ref{one dim direct sum in non-abelian group}$(2)$, $\tau_{irr}(G) \leq 4$. Again, using the spin character table of $G$, we conclude that $\tau_{irr}(G)= 4$.  \\

\textbf{Example 5.8.} 
Consider $G= S_6$. By Lemma \ref{one dim direct sum in non-abelian group}$(2)$, $\tau_{irr}(G) \leq 5$. Let $\psi$ be a faithful irreducible $\alpha$-representation of $G$ of degree $\tau_{irr}(G)$. If $\alpha= 1$, since $\di=5$, $\deg(\psi) \geq 5$. Let $\rho_6$ be the basic spin representation of $S_6^*$ (See \cite[5.5.B.]{karpilovsky3} for definition). It follows from \cite[Chapter 5, Corollary 5.8, Theorem 5.12]{karpilovsky3} that as a projective representation $\rho_6$ is irreducible and faithful on $S_6$ with $\deg(\rho_6)= 4$. Since any $\alpha$-representation has even degree for $[\alpha] \neq 1$, Example 5.1 yields that $\tau_{irr}(G)= 4$. Now, we have $\tau(G) \leq \tau_{irr}(G)= 4$. If $\rho$ is an ordinary representation of degree $m$ which is faithful as a projective representation, then $m \geq \delta(G)= 5$. Therefore, Example 5.1 implies that $\tau(G)= 4$. \\ 

\textbf{Example 5.9} 
$\tau(A_6)= \tau_{irr}(A_6)= 3$ follows from \cite[Table 2]{gonzalo} using Example 5.1 and Example 5.5.     \\ 

\textbf{Example 5.10.} 
Consider $G = A_7$. Since $S_5 \subseteq A_7$, $\tau(G) \geq 4$ in view of Example 5.7. Thus, $\tau(G) = \tau_{irr}(G) =4$ by virtue of \cite[Table 3]{gonzalo} and Example 5.5.    \\

\textbf{Example 5.11.} 
Consider $G = A_8 \cong \mathrm{PSL}_4(\F_2)$. By \cite[Lemma 5.2]{LANDAZURI}, $\tau(G) \geq 7$. Thus, $\tau(G) = \tau_{irr}(G) = 7$ due to Example 5.6 and Example 5.5.    \\

\textbf{Example 5.12.} 
Consider $G = A_9$. Since $A_8 \subset A_9$, $\tau(G) \geq 7$. Using Example 5.6, we obtain that $\tau(G) \leq 8$. Let $\rho$ be a faithful projective representation of degree $\tau(G)=m$. If $\deg(\rho)=7$, then $\rho$ has to be an ordinary representation. But $\delta(G) = 8$ and so $\tau(G) > 7$. Hence, it follows from Example 5.5 that $\tau(G) = \tau_{irr}(G)= 8$.   \\

The values of $\ti$ and $\tg$ obtained in the above examples are summarized in the following table. If $\tau_{irr}(G)$ does not exist for a group $G$, then we denote this by writing $\tau_{irr}(G)=$ -.  

\begin{table}[H]
	\centering
	\small
	\begin{tabular}{|c|c|c|}
		\hline
		Group & $\tau(G)$ & $\tau_{irr}(G)$ \\
		\hline
		$D_{2n}$ & $2$ & $2$ \\
		$Q_{4n}$ & $3$ & - \\
		$A_4$ & $2$ & $2$ \\
		$A_5$ & $2$ & $2$ \\
		$A_6$ & $3$ & $3$ \\
		$A_7$ & $4$ & $4$ \\
		$A_8$ & $7$ & $7$ \\
		$A_9$ & $8$ & $8$ \\
		$S_4$ & $2$ & $2$ \\
		$S_5$ & $4$ & $4$ \\
		$S_6$ & $4$ & $4$ \\
		\hline
	\end{tabular}
	\caption{Some examples} 
	\end{table}
	
	\section{$\tau(G)$ and $\tau_{irr}(G)$ for non-abelian groups of order $p^n$, $3\leq n \leq 5$} \label{groupsoforderp5}
	In this section, we compute $\tau(G)$ and $\tau_{irr}(G)$ for non-abelian $p$-groups of order $p^3$, $p^4$ for all primes $p$, and for groups of order $p^5$ for primes $p\geq 5$. 
The next theorem derives the minimum number of elements in $\mathrm{Irr}(G)$ such that their direct sum yields a faithful projective representation of $G$. 
\begin{lemma} \label{direct sum of k irreducible ordinary representations}
Let $G$ be a non-abelian $p$-group such that the rank of $Z(G)$ is $r$. Suppose $\rho_i \in \mathrm{Irr}(G)$ are such that $\oplus_{i=1}^k \rho_i$ is a faithful projective representation of $G$. 
\begin{enumerate}
	\item[(i)] If $Z(G) \not \cong (\Z/2\Z)^2, (\Z/2\Z)^4$, then $k >r$.
	\item[(ii)] If $Z(G) \cong (\Z/2\Z)^2, (\Z/2\Z)^4$, then $k >r-1$.
\end{enumerate}

\end{lemma}

\begin{proof}
(i) Let $\rho = \oplus_{i=1}^k\rho_i$, where $k \leq r$. Suppose $Z(G) \not \cong (\Z/2\Z)^2, (\Z/2\Z)^4$. Let $\chi_i \in \Hom(Z(G), \mathbb C^\times)$ such that $\rho_i|_{Z(G)}= \chi_i I_{n_i\times n_i}$. Then, $\rho|_{Z(G)}= \oplus_{i=1}^k\chi_i I_{n_i\times n_i}$.
Hence, if $\rho$ is faithful on $G$, then by Lemma \ref{restriction}$(2)$, $\oplus_{i=1}^k\chi_i$ is a faithful projective representation of $Z(G)$. Since $\tau(Z(G)) = r+1$ (by Theorem \ref{elementary abelian p-group}), we must have $k > r$. 

(ii) By similar arguments, it can be shown that if $Z(G) \cong (\Z/2\Z)^2, (\Z/2\Z)^4$, then $k >r-1$. This completes the proof. 

\end{proof}
We now state the next result from \cite[Theorem 3.1, Theorem 3.3]{kaur}.
\begin{theorem} \label{embedding degree of a p group with cyclic centre}
Let $G$ be a non-abelian $p$-group with cyclic centre. If $cd(G) =\{1,p^a\}$ or $\{1,p,p^a\}$ for $a \geq 1$, then $\delta(G)= \delta_{irr}(G)= p^a$. 
\end{theorem}
Recall that $A.G$ is a covering of $G$ if $\exists$ a subgroup $A \subseteq {(A.G)}' \cap Z(A.G) \cap M(G)$ such that $(A.G)/A \cong G$. 

\begin{theorem} 
\label{projective embedding degrees}
Let $B_G= \{A.G \mid A \cong Z(A.G) \text{ and } \delta_{irr}(A.G) \text{ exists} \}$. Then $\tau_{irr}(G)$ exists if and only if $B_G \neq \emptyset$, in which case $\tau_{irr}(G) = \min\{\delta_{irr}(A.G) \mid A.G \in B_G\}$. 
Furthermore, $\tau(G) \geq \min\{\delta(A.G) \} $.
\end{theorem}

\begin{proof}
Let $\rho$ be a faithful projective representation of $G$ and $\rho^*$ be a lift of $\rho$ in $G^*$. Suppose $H= G^*/K_{\rho^*}$. It is easy to see that $\rho^*$ induces a faithful ordinary representation of $H$.
Suppose $\tau_{irr}(G)$ exists. By Theorem \ref{projective embedding degrees}, $B_G \neq \emptyset$. Let $H= A.G \in B_G$. Since $\delta_{irr}(H)$ exists, $Z(H)$ is cyclic due to \cite[Theorem 2.32(a)]{isaacs}. Since $A \subseteq \mathrm{H}^2(G, \C^{\times})$ and $A = Z(H)$, we have $A \cong \Z/p\Z$. Hence, $|H|= p^{k+1}$.

Conversely, let $H$ be a covering of $G$ of order $p^{k+1}$ such that $G \cong H/Z(H)$. Since $Z(H)$ is cyclic, by \cite[Theorem 2.32(b)]{isaacs}, $\delta_{irr}(H)$ exists. Hence, $H \in B_G$. By Theorem \ref{projective embedding degrees}, $\tau_{irr}(G)$ exists.
From the above argument and Theorem \ref{projective embedding degrees}, it follows that $\tau_{irr}(G) = \min\{\delta_{irr}(H) \mid |H|= p^{k+1} \text{ s.t. } Z(H) \cong \Z/p\Z \}$. 

Now, let $\eta$ be a faithful projective representation of $G$ of degree $\tau(G)$ and $\eta^*$ be a lift of $\eta$ in $G^*$. Suppose $H= G^*/K_{\eta^*}$. It is easy to see that $\eta^*$ induces a faithful ordinary representation of $H$. By \cite[Proposition 2.15]{gonzalo}, $H$ is a covering of $G$. Hence, $\tau(G) \geq \min\{\delta(A.G) \}$. 

\end{proof}

\begin{corollary} \label{Elementary abelian Schur multiplier}
(i) Let $G$ be a non-abelian group of order $p^k$ and $\mathrm{H}^2(G, \C^{\times})$ be an elementary abelian $p$-group. Then, $\tau_{irr}(G)$ exists if and only if there exists a covering $H$ of order $p^{k+1}$ such that $G \cong H/Z(H)$. Furthermore, $\tau_{irr}(G) = \min\{\delta_{irr}(H) \mid |H|=p^{k+1} \text{ s.t. } Z(H) \cong \Z/p\Z \}$.\\
(ii) Let $G$ be a non-abelian $p$-group such that $cd(G^*)= \{1, p^n\}$. If $\tau_{irr}(G)$ exists, then $\tau(G)= \tau_{irr}(G)= p^n$. 
\end{corollary}

\begin{proof}
$(i)$ Suppose $\tau_{irr}(G)$ exists. By Theorem \ref{projective embedding degrees}, $B_G \neq \emptyset$. Let $H= A.G \in B_G$. Since $\delta_{irr}(H)$ exists, $Z(H)$ is cyclic due to \cite[Theorem 2.32(a)]{isaacs}. Since $A \subseteq \mathrm{H}^2(G, \C^{\times})$ and $A = Z(H)$, we have $A \cong \Z/p\Z$. Hence, $|H|= p^{k+1}$.

Conversely, let $H$ be a covering of $G$ of order $p^{k+1}$ such that $G \cong H/Z(H)$. Since $Z(H)$ is cyclic, by \cite[Theorem 2.32(b)]{isaacs}, $\delta_{irr}(H)$ exists. Hence, $H \in B_G$. By Theorem \ref{projective embedding degrees}, $\tau_{irr}(G)$ exists and also $\tau_{irr}(G) = \min\{\delta_{irr}(H) \mid |H|= p^{k+1} \text{ s.t. } Z(H) \cong \Z/p\Z \}$ by the above arguments. 

$(ii)$ Let $A.G$ be any covering of $G$. If $p^m \in cd(A.G)$, let $\Gamma \in \irr(A.G)$ be of degree $p^m$. Then, there is a projective representation $\rho$ of $G$ of degree $p^m$ which lifts to $\Gamma$. Since $cd(G^*)=\{1, p^n\}$, by Theorem \ref{bijective correspondence}, $cd(A.G)=\{1, p^n\}$. If $\tau_{irr}(G)$ exists, then by Theorem \ref{projective embedding degrees}, $B_G \neq \emptyset$. Let $H \in B_G$. Since $\delta_{irr}(H)$ exists, $Z(H)$ is cyclic. By Theorem \ref{embedding degree of a p group with cyclic centre}, $\delta_{irr}(H)= p^n$. Hence, $\tau_{irr}(G) = p^n$ due to Theorem \ref{projective embedding degrees}. By Lemma \ref{one dim direct sum in non-abelian group}$(1)$, $\tau(G)= p^n$. 
\end{proof}




We suggest that readers keep all the cited papers listed below handy, as they will be used without further reference.
\begin{enumerate}
\item The classification of groups $G$ of order $p^3$, $p^4$ and $p^5$ for an odd prime $p$ is given in \cite{james}. 

\item For the capability and Schur multiplier of groups of order $p^5$, $p\geq 5$, see \cite[Table 1 and Table 2]{hatuip5}.

\item The Schur multiplier of groups of order $p^4$ are given in \cite{Ellis} and \cite{OK} for odd primes $p$, and we use GAP (\cite{GAP4}) for $p=2$.

\item For non-abelian groups $G$ of order $p^4$ and $p^5$ with elementary abelian Schur multiplier, to determine all suitable groups $H$ satisfying the hypotheses of Corollary~\ref{Elementary abelian Schur multiplier}(i), we refer to \cite[Table~4.1]{james} for odd primes $p$, and to \cite{hall2n} for groups of order $2^4$.

\item We also require information on $\delta(G)$ and $\delta_{irr}(G)$ for the above-mentioned groups $G$, which are discussed in \cite{kaur}. If $G$ is a group of order $p^5$ or $p^6$, we refer to \cite[Table~4.1]{james} to determine $\mathrm{cd}(G)$. 
\end{enumerate}
\begin{theorem} \label{Groups of order $p^3$}
For non-abelian groups $G$ of order $p^3$,  $\tau(G)$ and $\tau_{irr}(G)$ are given below.
(i) $\tau_{irr}(\Phi_2(21))$ does not exist and $\tau(\Phi_2(21)) =p+1$. \\
(ii) $\tau_{irr}(\Phi_2(1^3)) = \tau(\Phi_2(1^3)) =p$. \\
(iii) $\tau_{irr}(Q_8)$ does not exist and  $\tau(Q_8) =3$.  \\
(iv) $\tau_{irr}(D_8) = \tau(D_8) =2$. 

\end{theorem}

\begin{proof} 
$(i)$ Let $G= \Phi_2(21)$. Since $G$ is extra-special, by Lemma \ref{capability of extra special $p$-group}, $G$ is not capable. So $\tau_{irr}(G)$ does not exist due to Corollary \ref{non-capable groups}.  
By \cite[Theorem 3.1(2)]{kaur}, $\delta(G)= p$. Therefore, Lemma \ref{relation} and Remark \ref{some results on p-groups}$(iii)$ yield that $\tau(G)= p+1$. \\
$(ii)$ For $G= \Phi_2(1^3)$, $\tau(G)= \tau_{irr}(G)= p$, by Theorem \ref{Heisenberg group}. \\ 
$(iii)$ The result follows from Example 5.4. \\
$(iv)$ The result follows from Example 5.1. \\
\end{proof}

\begin{theorem} \label{$p^4$}
For non-abelian groups $G$ of order $p^4$  ($p\geq 3)$, $\tau(G)$ and $\tau_{irr}(G)$ are given below. \\
(i) $\tau_{irr}(G)= p^2$ if $G$ is $\Phi_{2}(1^{4})$ or $\Phi_2(22)$ and $\tau_{irr}(\Phi_{3}(1^{4}))= \tau(\Phi_{3}(1^{4})) =p$. \\
(ii) $\tau_{irr}(G)$ exists if and only if $G$ is one of the groups: $\Phi_2(1^4)$, $\Phi_2(22)$ or $\Phi_3(1^4)$. \\
(iii) $\tau(G) = p+1$, if $Z(G)$ is cyclic and $G \ncong \Phi_3(1^4)$. \\
(iv) $\tau(G) = p+2$, if $Z(G) \cong \Z/p\Z \times \Z/p\Z$.

\end{theorem}

\begin{proof}
$(i)$ Let $G_1 = \Phi_{2}(1^{4})$. Then, $B_{G_1}= \{H \mid |H|= p^5, H \in \Phi_{7} \}$. If $H \in B_{G_1}$, then $\delta_{irr}(H)= p^2$ in view of Theorem \ref{embedding degree of a p group with cyclic centre}. Hence, $\tau_{irr}(G_1)= p^2$ due to Corollary \ref{Elementary abelian Schur multiplier}$(i)$.

Let $G_2 = \Phi_2(22)$ and $G_3 = \Phi_{3}(1^{4})$. Then, $B_{G_2}= \{H \mid |H|= p^5, H \in \Phi_{8} \}$ and $B_{G_3}= \{H \mid |H|= p^5, H \in \Phi_{9} \text{ or } \Phi_{10}\}$. Using similar arguments, we get $\tau_{irr}(G_2)= p^2$ and $\tau_{irr}(G_3)= p$. By Remark \ref{some results on p-groups}$(ii)$, $\tau(G_3)= p$. 



$(ii)$ If $G$ is one of the groups $\Phi_2(1^4)$, $\Phi_2(22)$ or $\Phi_3(1^4)$, then $\tau_{irr}(G)$ exists due to $(i)$. By \cite{hasanat}, these are the only three non-abelian groups that are capable. Hence, by Corollary \ref{non-capable groups}, $\tau_{irr}(G)$ does not exist for all other groups $G$.

$(iii)$ From $(ii)$ and Remark \ref{some results on p-groups}$(iii)$, it follows that $\tau(G)>p$. Since $\delta(G)= p$ if $Z(G)$ is cyclic (by \cite[Theorem 3.1(3)]{kaur}), $\tau(G)= p+1$ due to Lemma \ref{relation}. \\
$(iv)$ By Remark \ref{some results on p-groups}$(ii)$ and Lemma \ref{direct sum of k irreducible ordinary representations}, $\tau(G)>p+1$. Since $\delta(G)= p+1$, by Lemma \ref{relation}, $\tau(G) \leq p+2$. Hence, the result follows.  

\end{proof}
In the following result, we use the notations for the groups of order $2^4$ as given in \cite{ConradGroups16}.
\begin{theorem}
For non-abelian groups $G$ of order $2^4$, $\tau(G)$ and $\tau_{irr}(G)$ are given below: \\
(i) If $G$ is one of the groups $ \mathbb{Z}/8\mathbb{Z} \rtimes_{3} \mathbb{Z}/2\mathbb{Z}$, $\mathbb{Z}/8\mathbb{Z} \rtimes_{5} \mathbb{Z}/2\mathbb{Z}$, $Q_{8} \rtimes \mathbb{Z}/2\mathbb{Z}$, $(\mathbb{Z}/4\mathbb{Z}) \rtimes (\mathbb{Z}/4\mathbb{Z})$ or $Q_{8} \times \mathbb{Z}/2\mathbb{Z}$, then $\tau_{irr}(G)$ does not exist. \\
(ii) $\tau(D_{16})= \tau_{irr}(D_{16})=2$;  $\tau(Q_{16}) =3$ and $\tau_{irr}(Q_{16})$ does not exist. \\
(iii) If $G$ is one of the groups $D_{8} \times \mathbb{Z}/2\mathbb{Z}$ or $(\mathbb{Z}/2\mathbb{Z})^{2} \rtimes \mathbb{Z}/4\mathbb{Z}$, then $\tau_{irr}(G) = \tau(G) = 4$. \\
(iv) If $G$ is one of the groups $(\mathbb{Z}/4\mathbb{Z}) \rtimes (\mathbb{Z}/4\mathbb{Z})$ or $Q_{8} \times \mathbb{Z}/2\mathbb{Z}$, then $\tau(G)= 4$. \\
(v) If $G$ is one of the groups $\mathbb{Z}/8\mathbb{Z} \rtimes_{5} \mathbb{Z}/2\mathbb{Z}$, $\mathbb{Z}/8\mathbb{Z} \rtimes_{3} \mathbb{Z}/2\mathbb{Z}$ or $Q_{8} \rtimes \mathbb{Z}/2\mathbb{Z}$, then $\tau(G)=3$. 
\end{theorem}

\begin{proof} 
$(i)$ For $G = \mathbb{Z}/8\mathbb{Z} \rtimes_{3} \mathbb{Z}/2\mathbb{Z}$ or $ \mathbb{Z}/8\mathbb{Z} \rtimes_{5} \mathbb{Z}/2\mathbb{Z}$, $M(G)= 1$. Hence, by Lemma \ref{one dim direct sum in non-abelian group}$(2)$, $\tau_{irr}(G)$ does not exist. If $G= Q_{8} \rtimes \mathbb{Z}/2\mathbb{Z}$, then $\tau_{irr}(G)$ does not exist due to \cite[Corollary 7.15]{NG}. \\ 
If $G$ is one of the groups $(\mathbb{Z}/4\mathbb{Z}) \rtimes (\mathbb{Z}/4\mathbb{Z})$ or $Q_{8} \times \mathbb{Z}/2\mathbb{Z}$, then by \cite{hall2n}, there does not exist any group $H$ of order $2^5$ such that $H/Z(H) \cong G$. Hence, the result follows from Corollary \ref{Elementary abelian Schur multiplier}$(i)$. \\

$(ii)$ For $G = D_{16}$ or $Q_{16}$, the result follows from Example 5.1 and Example 5.4. \\ 

$(iii)$ Let $G_1 = D_{8} \times \mathbb{Z}/2\mathbb{Z}$ and $G_2 = (\mathbb{Z}/2\mathbb{Z})^{2} \rtimes \mathbb{Z}/4\mathbb{Z}$. Consider $H_1 \in \Gamma_6$ and $H_2 \in \Gamma_7$ of order $2^5$. By \cite{hall2n}, $G_i \cong H_i/Z(H_i)$, $i=1,2$, and so $\tau_{irr}(G_i)$ exists due to Theorem \ref{capable groups}. By Example 5.1, Lemma \ref{direct sum of k irreducible ordinary representations} and Theorem \ref{one dim direct sum in non-abelian group}$(1)$, $\tau(G_i) > 3$. Hence, $\tau(G_i)= \tau_{irr}(G_i)= 4$. \\   


$(iv)$ Suppose $G_3 = Q_{8} \times \Z/2\Z$. By Lemma \ref{direct sum of k irreducible ordinary representations} and Theorem \ref{one dim direct sum in non-abelian group}$(1)$, $\tau(G_2) > 3$.  Since $\delta(G_3)= 3$ (by \cite[Theorem 3.1(3)]{kaur}), $\tau(G_3) \leq 4$ due to Lemma \ref{relation}. Hence, $\tau(G_3)= 4$.  
Let $G_4 = D_{8} \times \mathbb{Z}/2\mathbb{Z}$. Since $(\Z/2\Z)^3 \subset G_4$, by Theorem \ref{elementary abelian p-group}, $\tau(G_4) \geqslant 4$. Also, as $\delta(G_4)= 3$, by Lemma \ref{relation}, $\tau(G_4) \leq 4$. Hence, $\tau(G_4)= 4$. \\  

$(v)$ 
Since $\delta(G)= 2$, the result follows from Example 5.2.

\end{proof}

In the next result, we study $\tau(G)$ and $\tau_{irr}(G)$ for non-abelian groups $G$ of order $p^5$, where $p \geq 5$.

\begin{theorem} \label{p5}
For non-abelian groups $G$ of order $p^5$, with $p \geq 5$, the following hold. 

(i) $G$ is not capable if and only if $G$ is one of the following groups: $\Phi_{2}(221)b, \Phi_{2}(311)a$, $ \Phi_{2}(311)b$, $\Phi_{2}(311)c$, $\Phi_{2}(32)a_1$, $\Phi_{2}(32)a_2$, $\Phi_{2}(41)$, $\Phi_{2}(2111)a$, $\Phi_{2}(2111)b$, $\Phi_{2}(2111)c$, \\ $\Phi_{2}(2111)d$, $\Phi_{3}(311)a, \Phi_{3}(311)b_r, \Phi_{3}(221)a, \Phi_{3}(2111)a$, $ \Phi_{3}(2111)b_r$, $\Phi_{3}(2111)c$, $\Phi_{3}(2111)d$, $\Phi_{3}(2111)e$, $\Phi_{4}(221)a, \Phi_{4}(221)c, \Phi_{4}(221)d_r \ (r \neq \tfrac{1}{2}(p-1))$, $\Phi_{4}(221)e$, $\Phi_{4}(221)f_r$, $\Phi_{4}(2111)a, \\ \Phi_{4}(2111)b, \Phi_{4}(2111)c$, $\Phi_{5}(2111)$, $\Phi_{5}(1^5)$, $\Phi_{6}(221)a$, $\Phi_{6}(221)b_r \ \big(r \neq \tfrac{1}{2}(p-1)\big)$, $\Phi_{6}(221)c_r$, $\Phi_{6}(221)d_r$, $\Phi_{6}(2111)a$, $\Phi_{6}(2111)b_r$, $\Phi_{7}(2111)a, \Phi_{7}(2111)b_r$, $\Phi_{7}(2111)c$, $\Phi_8(32)$, $\Phi_{9}(2111)a, \\ \Phi_{9}(2111)b_r$, $\Phi_{10}(2111)a_r$, $\Phi_{10}(2111)b_r$. 
For these groups, $\tau_{irr}(G)$ does not exist. \\



(ii) If $G = \Phi_{2}(221)c$, then $\ti = p^2$. \\

(iii) If $G = \Phi_{2}(1^5)$, then $\ti = p^2$. \\

(iv) If $G= \Phi_{2}(221)a$ or $\Phi_{2}(221)d$, then $\ti$ does not exist.   \\ 

(v) Let $G \in \Phi_{2}$. Then one of the following holds:
\begin{enumerate}[label=(\alph*)]
	\item $\tau(G) = p+1$, if $Z(G)$ is cyclic;
	\item $\tau(G) = p+2$, if $Z(G) \cong \Z/p^{2}\Z \times \Z/p\Z$;
	\item $\tau(G) = p+3$ if $Z(G) \cong \Z/p\Z \times \Z/p\Z \times \Z/p\Z$. 
	\\
\end{enumerate}

(vi) If $G = \Phi_{3}(1^5)$ or $\Phi_{3}(221)b_{r}$, then $\ti= p^2$.  \\

(vii) For $G \in \Phi_{3} $, the following hold:
\begin{enumerate}[label=(\alph*)]
	\item $\tau(G) = p+1$, if $Z(G)$ is cyclic;
	\item $\tau(G) = p+2$, if $Z(G) \cong \Z/p\Z \times \Z/p\Z$.  \\
\end{enumerate} 

(viii) If $G = \Phi_{4}(221)b$ or $\Phi_{4}(1^5)$, then $\ti=p^2$.  \\

(ix) If $G = \Phi_{4}(221)f_0$ or $\Phi_{4}(221)d_{\tfrac{1}{2}(p-1)}$, then $\ti$ does not exist. \\ 

(x) If $G \in \Phi_{4}$ or $\Phi_{5}$, then $\tg = 2p$. \\




(xi) If $G$ is one of the groups $\Phi_{6}(221)b_{\tfrac{1}{2}(p-1)}$, $\Phi_{6}(221)d_0$ or $\Phi_{6}(1^5)$, then $\ti= p^2$.  \\

(xii) If $G$ is one of the groups $\Phi_{6}(221)a$, $\Phi_{6}(221)b_r \ \big(r \neq \tfrac{1}{2}(p-1)\big)$, $\Phi_{6}(221)c_r$, $\Phi_{6}(221)d_r$, $\Phi_{6}(221)b_{\tfrac{1}{2}(p-1)}$, $\Phi_{6}(221)d_0$, $\Phi_6(2111)a$ or  $\Phi_6(2111){b_r}$, then $\tg = 2p+1$. \\

(xiii) If $G= \Phi_{6}(1^5)$, then $\tg = 2p$. \\


(xiv) If $G = \Phi_{7}(1^5)$, then $\ti= p^2$. \\

(xv) If $G \in \Phi_{7}$, then $\tg = 2p$. \\


(xvi) If $G = \Phi_8(32)$, then $\tau(G)=p^2+1$. \\


(xvii) If $G= \Phi_{9}(2111)a$ or $\Phi_{9}(2111)b_r$, then $\tau(G) = p+1$. \\

(xviii) If $G = \Phi_{9}(1^5)$, then $\ti = \tg = p$. \\

(xix) If $G = \Phi_{10}(2111)a_r$ or $\Phi_{10}(2111)b_r$, then $\tg= p^2+1.$ \\ 

(xx) If $G = \Phi_{10}(1^5)$, then $\ti = \tg= p^2$.

\end{theorem}

\begin{proof}

$(i)$ By considering the list of capable groups given in \cite[Table 2]{hatuip5}, the result follows from Corollary \ref{non-capable groups}.  \\

$(ii)$ Suppose $G= \Phi_{2}(221)c =\langle \alpha, \alpha_1, \alpha_2, \gamma \mid [\alpha_1,\alpha] =\gamma^p= \alpha_2,\; \alpha^{p^2}= \alpha_1^{p} = \alpha_2^{p}=1 \rangle$. 
Then the group 
\begin{eqnarray*}
	G^* &=& \langle \Bar{\alpha}, \Bar{\alpha_i}, \Bar{\gamma}, \gamma_i, \mid [\Bar{\alpha_1}, \Bar{\alpha}] = \Bar{\gamma}^p = \Bar{\alpha_2},\ [\Bar{\gamma}, \Bar{\alpha_1}] = \gamma_1,\ \\
	&& [\Bar{\gamma}, \Bar{\alpha}] = \gamma_2,\ \Bar{\alpha}^{p^2}= \Bar{\alpha_i}^{p} = \gamma_1^{p}= \gamma_2^{p^2}= 1 \ (i = 1,2) \rangle  
\end{eqnarray*}
is a representation group of $G$ such that ${G}^*/\langle \gamma_1, \gamma_2 \rangle \cong G$. Let $A= \langle \gamma_1, \gamma_2 \rangle$.
The proof uses the same argument as in the proof of Theorem~\ref{Heisenberg group}(1).
By \cite[Theorem 9.4]{Mackey}, it follows that any cocycle of $G$ is cohomologous to a cocycle $\nu$ of the form
$$
\nu\big(\gamma^{m_1}\alpha^{n_1}\alpha_1^{r_1}, \gamma^{m_2}\alpha^{n_2}\alpha_1^{r_2}\big) = \epsilon^{(m_2+p n_2r_1)n_1+\frac{r_1pn_2(n_2-1)}{2}}\lambda^{r_1m_2},
$$ for some $\epsilon$, $\lambda \in \C^\times$ such that $\epsilon^{p^2}= \lambda^{p}= 1$. It is easy to check that $1$ is the only $\nu$-regular element of $Z(G)$ if and only if $\nu(\alpha, \gamma)= \epsilon$ is a primitive $p^2$th root of unity. By Lemma \ref{faithful projective representation}, $G$ admits a faithful irreducible $\nu$-representation only for such $\nu$, and let $S$ be the set of all such $\nu$. Then $\tau_{irr}(G)= \operatorname{min}\{\operatorname{deg}(\operatorname{\rho}) \mid \rho \in \mathrm{Irr}^\nu(G), \nu \in S\}$. By Theorem \ref{bijective correspondence}, it follows that $$\tau_{irr}(G)=\operatorname{min}\{\operatorname{deg}(\operatorname{\Gamma}) \mid \Gamma \in \mathrm{Irr}(G^*\mid \chi), \chi\in \Hom(A, \C^\times) \text{ such that } \operatorname{tra}(\chi) \in S\}.$$ 
Let $T= \{\chi\in \Hom(A, \C^\times)\mid \chi(\gamma_2) \text{ is a primitive } p^2 \text{th root of unity}\}.$
First, we claim that $\operatorname{tra}(\chi) \in S$ if and only if $\chi\in T$. If $\chi^p= 1$, then $\operatorname{tra}(\chi)^p= 1$.  
Hence, if $\operatorname{tra}(\chi) \in S$, then $\chi\in T$. Conversely, let $\chi\in T$ and $\mu: G \to G^*$ be a section of $G^* \twoheadrightarrow G$ defined by $\mu(\gamma^{m} \alpha^{n} \alpha_1^{r}) = \Bar{\gamma}^{m} \Bar{\alpha}^{n} \Bar{\alpha_1}^{r}$. Then, 
$$\operatorname{tra}(\chi)(\alpha, \gamma) = \chi(\mu(\alpha)\mu(\gamma)\mu(\alpha\gamma)^{-1})=\chi(\Bar{\alpha}\Bar{\gamma}(\Bar{\gamma} \Bar{\alpha})^{-1})=\chi(\gamma_2^{-1}).$$ Thus, $\operatorname{tra}(\chi) \in S$ and the claim is proved.
Therefore, $$\tau_{irr}(G)=\operatorname{min}\{\operatorname{deg}(\operatorname{\Gamma}) \mid \Gamma \in \mathrm{Irr}(G^*\mid \chi),  \chi \in T\}.$$

Let $\chi \in T$ and $\Gamma \in \mathrm{Irr}(G^*\mid \chi)$. We have $\chi(\gamma_2)= \xi$,  where $\xi$ is a primitive $p^2$th root of unity. 
Now we 
show that $\operatorname{deg}(\Gamma)$ cannot be $p$.
Consider the normal abelian subgroup $N= \langle \Bar{\gamma}, \gamma_1, \gamma_2 \rangle$ of $G^*$ of order $p^5$ and $\chi_1 \in \Hom(N, \C^\times)$ such that $\chi_1|_A= \chi$. Let $\Gamma \in \mathrm{Irr}(G^*\mid \chi_1)$.
If $\Bar{\alpha_1} \notin {I_{{G}^*}(\chi_1)}$ and $\operatorname{deg}(\Gamma)=p$, then by  Clifford's theory, $\Bar{\alpha} \in {I_{{G}^*}(\chi_1)}$, which is not possible. Hence, $\Bar{\alpha_1} \in {I_{{G}^*}(\chi_1)}$. Then, $\chi_1$ extends to $\chi_2: N_1 \to \C^\times$, where $N_1= \langle N, \Bar{\alpha_1} \rangle$. Since $\chi_2^{(\Bar{\alpha}^p)}(\Bar{\gamma})\neq \chi_2(\Bar{\gamma})$, hence $\operatorname{deg}(\Gamma)$ cannot be $p$.  
Therefore, $\tau_{irr}(G)= p^2$. \\



$(iii)$  Let $G= \Phi_{2}(1^{5})$. Then, $B_{G}=\{H \mid |H|= p^6, H \in \Phi_{22}\}$. If $H \in B_{G}$, then $\delta_{irr}(H)= p^2$ in view of Theorem \ref{embedding degree of a p group with cyclic centre}. Hence, $\tau_{irr}(G)= p^2$ due to Corollary \ref{Elementary abelian Schur multiplier}$(i)$. \\


$(iv)$ If $G= \Phi_{2}(221)a$ or $\Phi_{2}(221)d$, it is easy to check that there does not exist any covering $H$ of order $p^6$ in \cite[Table 4.1]{james} such that $H/Z(H) \cong G$. Hence, the result follows from Corollary \ref{Elementary abelian Schur multiplier}$(i)$. \\ 

$(v)$ Let $G \in \Phi_2$. 
\begin{enumerate}[label=(\alph*)]
	\item If $Z(G)$ is cyclic, then $\delta(G)= p$, by \cite[Theorem 3.1(4)(i)]{kaur}. The result now follows from Remark \ref{some results on p-groups}$(iii)$. 
	\item  If $Z(G) \cong \Z/p^{2}\Z \times \Z/p\Z$, then $\delta(G)= p+1$ and the result follows from Remark \ref{some results on p-groups}$(iii)$, Lemma \ref{direct sum of k irreducible ordinary representations}. 
	\item If $Z(G) \cong \Z/p\Z \times \Z/p\Z \times \Z/p\Z$, then $\delta(G)= p+2$. Using arguments similar to those in $(ii)$, the result follows. 
\end{enumerate}

$(vi)$  Let $G_1 = \Phi_{3}(1^{5})$. Then, $B_{G_1}= \{H \mid  |H|= p^6, H \in \Phi_{24} \text{ or } \Phi_{27}\}$. If $H \in B_{G_1}$, then $\delta_{irr}(H)= p^2$ in view of Theorem \ref{embedding degree of a p group with cyclic centre}. Hence, $\tau_{irr}(G_1)= p^2$ due to Corollary \ref{Elementary abelian Schur multiplier}$(i)$.
Let $G_2 = \Phi_{3}(221)b_{1}$ and $G_3 = \Phi_{3}(221)b_{\nu}$. Then $B_{G_2}= \{H \mid |H|= p^6, H \in \Phi_{25} \text{ or } \Phi_{28}\}$ and $B_{G_3}= \{H \mid |H|= p^6, H \in \Phi_{26} \text{ or } \Phi_{29} \}$.
By similar arguments, it can be shown that $\tau_{irr}(G_2)=\tau_{irr}(G_3)= p^2$. \\




$(vii)$ 
If $G \in \Phi_{3}$, then $cd(G)=\{1,p\}$ by \cite[Lemma 5.3]{Prajapati}. \\
$(a)$ If $Z(G)$ is cyclic, then by \cite[Theorem 2.32(b)]{isaacs}, $G$ has a faithful irreducible ordinary representation of degree $p$. So, $\delta(G)= p$, in view of Lemma \ref{one dim direct sum in non-abelian group}$(1)$. By Lemma \ref{relation} and Remark \ref{some results on p-groups}$(iii)$, $\tau(G) = p+1$. 
$(b)$ If $Z(G)\cong \Z/p\Z \times \Z/p\Z$, then by Remark \ref{some results on p-groups}$(iii)$ and Lemma \ref{direct sum of k irreducible ordinary representations}, we have $\tau(G) > p+1$. 
Consider the group $$G= \Phi_{3}(221)a= \langle \alpha, \alpha_1, \alpha_2, \alpha_3
\mid [\alpha_1,\alpha]= \alpha_2,\; [\alpha_2,\alpha]= \alpha^{p}= \alpha_3,\;
\alpha_1^{p^2}= \alpha_2^{p}= \alpha_3^{p} =1 \rangle.$$ Then $Z(G)= \langle \alpha_3,\; \alpha_1^{p} \rangle$. Let $\chi \in \operatorname{Hom}(Z(G),\mathbb{C}^{\times})$ be defined on the generators by $\chi(\alpha_1^{p}) = 1$, $\chi(\alpha_3)= \xi$, where $\xi$ is a primitive $p$th root of unity. Let $\rho \in \operatorname{Irr}(G \mid \chi)$. As $\alpha_3 \in Z(G) \cap G^\prime$ and $\chi(\alpha_3)\neq 1$, we have $\operatorname{deg}(\rho)= p$. 
Now, let $\chi_1 \in \operatorname{Hom}(G,\mathbb{C}^{\times})$ be defined on the generators by $\chi_1(\alpha_1)= \xi_1$, $\chi_1(\alpha)=1$, where $\xi_1$ is a primitive $p^2$th root of unity. Since $\chi \oplus \chi_1|_{Z(G)} \oplus 1_{Z(G)}$ is faithful on $Z(G)$, $\rho \oplus \chi_1 \oplus 1_G$ is faithful on $G$, follows from Lemma \ref{restriction}$(2)$. Hence, $\tau(G) = p+2$.  
Similarly, for all other groups $G$ whose centre is not cyclic, it can be shown that $\tg = p+2$. \\

$(viii)$ Let $G_1 = \Phi_{4}(221)b$. Then, $B_{G_1}= \{H \mid H \in \Phi_{34}, |H|= p^6 \}$. If $H \in B_{G_1}$, then $\delta_{irr}(H)= p^2$ by Theorem \ref{embedding degree of a p group with cyclic centre}. Hence, using Corollary \ref{Elementary abelian Schur multiplier}$(i)$, $\tau_{irr}(G_1)= p^2$.
Let $G_2 = \Phi_{4}(1^5)$. Then, $B_{G_2} = \{H \mid H \in \Phi_{31}, \Phi_{32} \text{ or } \Phi_{33}, |H|= p^6\}$. A similar argument shows that $\tau_{irr}(G_2)= p^2$. \\

$(ix)$
Consider the group $$G_1= \Phi_4(221){f_0}
= \langle \alpha, \alpha_i,  \beta_i \mid [\alpha_i, \alpha] = \beta_1,\ 
\alpha_1^p = \beta_2,\ \alpha_2^p = \beta_1^{\,\nu}, \ \alpha^p = \beta_1^p = 1 \ (i = 1,2) \rangle.$$ Then $G_1^* = \langle \alpha, \alpha_i,  \beta_i, \gamma \mid [\alpha_i, \alpha] = \beta_i,\ [\alpha_1, \alpha_2] = \gamma,\ \alpha_1^p = \beta_2,\ \alpha_2^p = \beta_1^{\,\nu}, \ \alpha^p = \beta_i^p = \gamma^{p^2}= 1 \ (i = 1,2) \rangle$ is a representation group of $G_1$ and $Z(G_1^*) = \langle \gamma, \beta_1, \beta_2 \rangle$. Since $M(G_1) \cong \Z/p^2\Z$, $\tau_{irr}(G_1)$ does not exist due to \cite[Chapter 4, Theorem 3.1]{karpilovsky2}. 
Next, consider the group $$G_2 = \Phi_4(221)d_{\tfrac{1}{2}(p-1)} = \langle \alpha, \alpha_i,  \beta_i \mid [\alpha_i, \alpha] = \beta_i,\ \alpha_1^p = \beta_1^{-1},\ \alpha_2^p = \beta_2,\ \alpha^p = \beta_i^p = 1 \ (i = 1,2) \rangle.$$ Then, $G_2^* = \langle \alpha, \alpha_i,  \beta_i, \gamma \mid [\alpha_i, \alpha] = \beta_i,\ [\alpha_1, \alpha_2] = \gamma,\ \alpha_1^p = \beta_1^{-1},\ \alpha_2^p = \beta_2,\ \alpha^p = \beta_i^p = \gamma^{p^2}= 1 \ (i = 1,2) \rangle$ is a representation group of $G_2$ and $Z(G_2^*) = \langle \gamma, \beta_1, \beta_2 \rangle$. Using the same argument as above, $\tau_{irr}(G_2)$ does not exist. \\


$(x)$ \label{Phi4} If $G \in \Phi_{4}$, then $Z(G) \cong \Z/p\Z \times \Z/p\Z$ and $G/Z(G) \cong (\Z/p\Z)^3$. By Lemma \ref{direct sum of k irreducible ordinary representations}, the direct sum of $p$-dimensional and one-dimensional ordinary representations of $G$ can not be faithful as projective. Hence, $\tau(G) >p+1$, in view of Remark \ref{some results on p-groups}$(iii)$. Since $Z(G) \subseteq G^\prime$,  $\psi|_{Z(G)}= 1$ for any $\psi \in \Hom(Z(G), \C^\times)$. This forces $\tau(G) >2p-1$. 
Now, Lemma \ref{exact sequence} yields the following exact sequence: \\

$1 \to \mathrm{Hom}(Z(G), \mathbb{C}^{\times}) 
\xrightarrow{\mathrm{tra}} 
\mathrm{H}^2(G/Z(G), \mathbb{C}^{\times}) 
\xrightarrow{\mathrm{inf}} \mathrm{H}^2(G, \mathbb{C}^{\times})$. \\

Let $\alpha$ be a non-trivial cocycle in $\operatorname{Im} \operatorname{(inf)}$, and let $\chi_i \in\mathrm{Hom}(Z(G), \mathbb{C}^{\times})$, $i=1,2$, be such that their direct sum $\chi_1 \oplus \chi_2$ is faithful on $Z(G)$. Let $\tilde{\rho_i} \in \mathrm{Irr}^{\beta \mathrm{tra}(\chi_i)}\big(G/Z(G)\big)$, $i=1,2$, and $\rho_i \in \mathrm{Irr}^{\alpha}(G)$ correspond to $\tilde{\rho_i}$ as described in Remark \ref{remark for bijective correspondence via inflation}. As $|G/Z(G)|= p^3$, so $\operatorname{deg}(\tilde{\rho_i})= p$. Now, by Theorem \ref{direct_sum}, $\rho_1 \oplus \rho_2$ is faithful of degree $2p$. Hence, $\tau(G)= 2p$.
If $G \in \Phi_{5}$, the result follows directly from Theorem \ref{extraspecial $p$-group}. \\

$(xi)$ Let $G_1 = \Phi_{6}(221)b_{\tfrac{1}{2}(p-1)}$. Then, $B_{G_1}= \{H \mid |H|= p^6, H \in \Phi_{42}\}$. If $H \in B_{G_1}$, then $\delta_{irr}(H)=p^2$ in view of Theorem \ref{embedding degree of a p group with cyclic centre}. Hence, $\tau_{irr}(G_1)= p^2$ due to Corollary \ref{Elementary abelian Schur multiplier}$(i)$.
If $G = \Phi_{6}(221)d_0$ or $\Phi_{6}(1^5)$, then by similar arguments, $\tau_{irr}(G)= p^2$. \\




$(xii)$ First, we show that if  $G\in \Phi_6$, then  $2p-1 < \tau(G) \leq 2p+1$. Proceeding along the same lines as the proof of $(x)$, we get $\tau(G) > 2p-1$. Now, let $\chi_1$ and $\chi_2$ be two non-trivial, distinct elements of $\Hom(Z(G), \C^\times)$ such that $\chi_1 \oplus \chi_2$ is a faithful ordinary representation of $Z(G)$. Let $\rho_i \in \operatorname{Irr}({G} \mid \chi_i)$, $i=1,2$. Since $cd(G)=\{1,p\}$ by  \cite[Table 4.1]{james}, $\operatorname{deg}(\rho_i)=p$. 
From Lemma \ref{restriction}(2), it follows that $(\oplus_{i=1}^2 \rho_i)\oplus 1_G$ is faithful on $G$. Hence, $\tau(G) \leq 2p+1$. 
By Lemma \ref{direct sum of k irreducible ordinary representations}, $G$ cannot possess an ordinary representation of degree $2p$ which is faithful. 

If $G$ is one of the groups $\Phi_{6}(221)a$, $\Phi_{6}(221)b_r \ \big(r \neq \tfrac{1}{2}(p-1)\big)$, $\Phi_{6}(221)c_r$ or $\Phi_{6}(221)d_r$, then $M(G)= 1$. Hence, $\tau(G) = 2p+1$.

Now, for other groups, it is enough to show that $G$ cannot possess a faithful $\alpha$-representation of degree $2p$ for any non-trivial cocycle $\alpha$.
\\
Let $G= \Phi_{6}(221)b_{\tfrac{1}{2}(p-1)}$ or $\Phi_{6}(221)d_0$ and $\rho \in \mathrm{Irr}^\alpha(G)$. By the proof of \cite[Theorem 4.3, $(xvii)$]{sabnam}, it follows that $\operatorname{deg}(\rho)= p^2$. Hence, the result holds. \\
Suppose $G_1 = \Phi_{6}(2111)a= \langle \alpha_1, \alpha_2, \beta, \beta_1, \beta_2
\mid [\alpha_1,\alpha_2]= \beta,\ [\beta,\alpha_i]= \beta_i,\ \alpha_1^p= \beta_1,\ \alpha_2^p= \beta^p= \beta_i^p= 1\ (i=1,2) \rangle$. From the proof of \cite[Theorem 4.3, (xvi)]{sabnam}, it follows that
\begin{eqnarray*}
	{G}_1^* & = \langle \alpha_1, \alpha_2, \beta, \beta_i,\ 1 \le i \le 3 \mid [\alpha_1,\alpha_2]  =\beta,\ [\beta,\alpha_j]= \beta_j,\ \\
	& [\beta_2,\alpha_2]= \beta_3,\ \alpha_1^p= \beta_1,\ \alpha_2^p = \beta^p = \beta_i^p =1\ (j=1,2) \rangle    
\end{eqnarray*}
is a representation group of $G_1$ of order $p^6$, and ${G}_1^*/\langle \beta_3 \rangle \cong G_1$. 
Now consider the abelian normal subgroup
$N_1= \langle \beta, \beta_i,\ 1 \le i \le 3 \rangle$ of ${G}_1^*$ of order $p^4$. 
Taking $A_1=\langle \beta_3 \rangle$, we have an isomorphism $\tra: \Hom(A_1,\C^\times) \to \mathrm{H}^2(G_1, \C^\times)$. Let $\chi' \in \mathrm{Hom}(\langle \beta_3 \rangle, \C^\times)$   such that $\chi'\neq 1$ and $\tra(\chi')=[\alpha]$. Suppose $\rho_i \in \mathrm{Irr}^\alpha(G_1)$ correspond to $\Gamma_i \in \operatorname{Irr}({G}_1^* \mid \chi')$ via the correspondence given in Theorem \ref{bijective correspondence}. Now, we show that if deg$(\rho_i) =p$, then $\rho_1 \oplus \rho_2$ is not faithful on $G_1$.
Suppose $\chi \in \mathrm{Irr}(N_1)$ such that $\chi|_{\langle \beta_3 \rangle}= \chi'$.
Let $\chi(\beta)=\xi^{i_1},\ \chi(\beta_1)=\xi^{i_2},\ 
\chi(\beta_2)=\xi^{i_3},\ \chi(\beta_3)=\xi^{i_4}$,
where $\xi$ is a primitive $p$th root of unity and
$0 \le i_1,i_2,i_3,i_4 \le (p-1)$. 
Since $i_4 \ne 0$ and deg$(\Gamma_i)= p$, from the proof of \cite[Theorem 4.3, $(xvi)$]{sabnam}, it follows that $i_2= 0$. In this case, $H_1 = I_{{G}_1^*}(\chi)= \langle N_1,\alpha_1\rangle$ is of order $p^5$. Let $\chi_1$, $\chi_2 \in \operatorname{Irr}(N_1)$ be such that $\chi_i|_{\langle \beta_3 \rangle}= \chi'$ and $\chi_i(\beta_1)=1$ for $i=1,2$. Then, $\Gamma_i = \operatorname{Ind}^{G_1^*}_{I_{G_1^*}(\chi_i)}(\tilde{\chi}_i)$, where $\tilde{\chi}_i:I_{G_1^*}(\chi_i)\to \C^\times$ is an extension of $\chi_i$. Choosing $t_i= \alpha_2^i$ for $i=1,2,..,p$, as left coset representatives of $H_1$ in $G_1^*$, we get $\Gamma_i(\beta_1)= I_{p \times p}$. Hence, $(\Gamma_1 \oplus \Gamma_2)(\beta_1)= I_{2p \times 2p}$ and so $\rho_1 \oplus \rho_2|_{Z(G_1)}$ is not faithful. Consequently, Lemma \ref{restriction}$(2)$ yields that $\tau(G_1)= 2p+1$. \\
Suppose $G_2= \Phi_6(2111){b_r}$. It follows from  \cite[Proof of Theorem 4.1(xvi)]{sabnam} that $(G_1,G_2)$ satisfies the hypotheses of Theorem \ref{tau(G_1) = tau(G_2)}. Hence, $\tau(G_2)= \tau(G_1)= 2p+1$. \\



(xiii) Let $G = \Phi_{6}(1^5) = \langle \alpha_1, \alpha_2, \beta, \beta_i \mid 
[\alpha_1,\alpha_2] = \beta,\,
[\beta,\alpha_i] = \beta_i,\,
\alpha_i^p = \beta^p = \beta_i^p = 1 \ (i=1,2) \rangle$. Now consider the following group 
\begin{eqnarray*}
	G^*&=&\langle \alpha_1, \alpha_2, \gamma ,\beta, \beta_j, 1\leq j \leq 4  \mid 
	[\alpha_1,\alpha_2] = \beta,\, 
	[\beta,\alpha_i] = \beta_i,\, [\beta_2,\alpha_2] = \beta_3,\, \\
	&& [\beta_1,\alpha_1] = \beta_4,\,
	[\beta_1,\alpha_2] = [\beta_2,\alpha_1] = \gamma,\,
	\alpha_i^p = \beta^p = \beta_j^p = \gamma^p = 1 \ (i=1,2) \rangle \\
	&=& \big(\langle  \alpha_2, \gamma ,\beta, \beta_j, 1\leq j \leq 3  \mid 
	[\beta,\alpha_2] = \beta_2,\, [\beta_2,\alpha_2] = \beta_3,\, [\beta_1,\alpha_2] =  \gamma,\, \\
	&& \alpha_2^p = \beta^p = \beta_j^p = \gamma^p = 1 \ \rangle \times \langle \beta_4 \rangle\big) \rtimes \langle \alpha_1 \rangle \\
	&=&\big(\Phi_{16}(1^6) \times  \Z/p\Z \big) \rtimes  \Z/p\Z.
\end{eqnarray*}
Then $G^*$ is a representation group of $G$ of order $p^8$ such that $G^*/Z(G^*)= G$, where $Z(G^*)= \langle \beta_3, \beta_4, \gamma \rangle$. 
Consider the abelian normal subgroup $N =\langle \beta, \gamma, \beta_i, 1 \leq i \leq 4 \rangle$ of ${G}^*$ of order $p^6$. Take $A= Z(G^*)$. Let $\chi' \in \mathrm{Irr}(A)$ such that $\chi'(\beta_3) = \xi$, $\chi'(\beta_4) = \xi$ and $\chi'(\gamma) = \xi$ where $\xi$ is a primitive $p$th root of unity. First, our aim is to find $\Gamma_i \in\operatorname{Irr}(G^* \mid \chi')$, $i=1,2$, of degree $p$. 
We use Clifford's theory here. Let $\chi \in \operatorname{Irr}(N)$ such that
$\chi|_{A} = \chi'$. Now $g \in {G}^*$ can be written as $g = \alpha_1^m \alpha_2^n n'$
for some $n' \in N$.
Now we have
$$\chi^{(\alpha_1^m \alpha_2^n)}(\beta^j \beta_1^k \beta_2^l z) = \chi(\beta^j \beta_1^k \beta_2^l z)
\iff \chi(\alpha_1^m \alpha_2^n \beta^j \beta_1^k \beta_2^l \alpha_2^{-n} \alpha_1^{-m}) = \chi(\beta^j \beta_1^k \beta_2^l).$$
Upon simplification, we obtain that
$\alpha_1^m \alpha_2^nn' \in I_{{G}^*}(\chi)$
if and only if
\[
\chi(\beta_1^{-jm} \beta_2^{-jn} \beta_3^{j\binom{p-n}{2}-ln} \beta_4^{-km+j\binom{p-m}{2}} \gamma^{-jnm-lm-kn}) = 1.
\]
Let $\chi(\beta_1) = 1$ and $\chi(\beta_2) = \xi^r$ where $0 < r \le (p-1)$. Then,
\[\xi^{-rjn+(j\binom{p-n}{2}-ln)+(-km+j\binom{p-m}{2})-jnm-lm-kn} = 1.\] 
Taking $j=0$, $k=0$, $l=1$, we have $\xi^{-n-m} = 1$, which implies $n= -m$.
Hence, $N_1= I_{{G}^*}(\chi) = \langle N, \alpha_1^{-2^{-1}r}\alpha_2^{2^{-1}r} \rangle= \langle N, \alpha_1^{-1} \alpha_2 \rangle$ is of order $p^7$. Let $\chi_1$ and $\chi_2 \in \operatorname{Irr}(N)$ be such that $\chi_i(\beta_1)= 1$ and $\chi_1(\beta_2)= \xi$, $\chi_2(\beta_2)= \xi^2$, $\chi_i|_{A} = \chi'$ for $i=1,2$. 
Since $I_{G^*}(\chi_i)/N$, $i=1,2$ is cyclic, $\chi_i$ extends to $\tilde{\chi}_i: I_{G^*}(\chi_i)\to \C^\times$. Consider $\Gamma_i = \operatorname{Ind}^{G^*}_{I_{G^*}(\chi_i)}(\tilde{\chi}_i)$. Let $\{l_t= \alpha_2^t, t=1,2,..,p\}$ be a set of left coset representatives of $I_{G^*}(\chi_i)$ in $G^*$. Then $\Gamma_i(\beta_1)= D_{p \times p}$, $\Gamma_1(\beta_2)= \xi{D}_{p \times p}$ and $\Gamma_2(\beta_2)= \xi^2 D_{p \times p}$, where $D_{p \times p}$ is a diagonal matrix such that $D_{ii}= \xi^{i}$. Let $\alpha= \operatorname{tra}(\chi')$ and $\rho_i \in \mathrm{Irr}^\alpha(G)$ correspond to $\Gamma_i$, where the correspondence is given by Theorem \ref{bijective correspondence}. Let $\rho= \rho_1 \oplus \rho_2$. 
Then, for $0\leq m, n \leq p-1$,
$$\rho(\beta_1^m \beta_2^n)= \left[\begin{array}{cccc|cccc}
	\xi^{m}\xi^{n}\xi^{n} & & & & & & &\\
	& \xi^{2m}\xi^{n}\xi^{2n} & & & & & & \\
	& & \ddots & & & & & \\
	& & & \xi^{n} & & & & \\
	\hline
	& & & & \xi^{m}\xi^{2n}\xi^{n} & & & \\
	& & & & & \xi^{m}\xi^{2n}\xi^{2n} & &  \\
	& & & & & & \ddots & \\
	& & & & & & & \xi^{2n}  \\
\end{array}\right]_{2p \times 2p}$$
is a diagonal matrix.
Since $\rho|_{Z(G)}$ is faithful, it follows from Theorem \ref{restriction} that $\rho$ is faithful. Hence, $\tau(G) = 2p$ as $2p-1 < \tau(G) \leq 2p+1$ for $G\in \Phi_6$ (as shown in the proof of $(xii)$). \\

$(xiv)$ Let $G= \Phi_{7}(1^{5})$. Then, $B_{G}=\{H \mid |H|= p^6, H \in \Phi_{30}\}$. If $H \in B_{G}$, then $\delta_{irr}(H)= p^2$ in view of Theorem \ref{embedding degree of a p group with cyclic centre}. Hence, $\tau_{irr}(G)= p^2$ due to Corollary \ref{Elementary abelian Schur multiplier}$(i)$.  \\

$(xv)$ \label{Phi7}
If $G \in \Phi_{7}$, then $Z(G)\cong \Z/p\Z$ and $G/Z(G) \cong \Phi_2(1^4)$. Let $G_1= G/Z(G)$. Then 
\[
G_1 \cong \langle \alpha, \alpha_i, 1\leq i \leq 3 \mid [\alpha_1,\alpha] = \alpha_2,\; \alpha^{p}= \alpha_i^{p} =1 \rangle.
\]
Using \cite[Theorem 9.4]{Mackey}, it is easy to check that every cocycle of $G_1$ is cohomologous to a cocycle $\gamma$ of the form \[
\gamma\big(\alpha_3^{\omega_1} \alpha_2^{x_1} \alpha^{y_1} \alpha_1^{z_1},\alpha_3^{\omega_2} \alpha_2^{x_2} \alpha^{y_2} \alpha_1^{z_2}\big) = \lambda^{x_2 z_1+\frac{y_2 z_1\left(z_1-1\right)}{2}}\mu^{ y_1 x_2+\frac{z_1 y_2(y-1)}{2}+z_1 y_1 y_2}\tau_y^{y_1 \omega_2} \tau_z^{z_1 \omega_2},\] where $\lambda,\mu,\tau_y,\tau_z \in \C^\times$ such that $\lambda^p= \mu^p = \tau_y^p = \tau_z^p= 1.$ 
Let $\lambda= e^{\frac{2 \pi i j}{p}}, \mu= e^{\frac{2 \pi i k}{p}}, \tau_y= e^{\frac{2 \pi i l}{p}}, \tau_z= e^{\frac{2 \pi i m}{p}}$. First, we show that if $\left[\begin{array}{l}j \\ k\end{array}\right]$ and $\left[\begin{array}{l}m \\  l \end{array}\right]$ are linearly dependent over $\Z/p\Z$, then $Z(G_1)$ has a non-trivial $\gamma$-regular element. If $(j,k)=n(m,l)$ for some $n\in \Z/p\Z$, then it is easy to check that $\alpha_3^{-n} \alpha_2$ is a $\gamma$-regular element of $Z(G_1)$. If $(m,l)=(0,0)$, then $\alpha_3$ is a $\gamma$-regular element of $Z(G_1)$. Hence, for any such $\gamma$, all irreducible $\gamma$-representations of $G_1$ are of degree $p$, in view of \cite[Chapter 2, Lemma 5.1]{karpilovsky3}. This implies  
that, if $X= \{[\nu] \in \mathrm{H}^2(G_1,\C^\times) \mid  \irr^\nu(G_1) \text{ contains an element of degree } p^2\}$, then $|X| \leq (p^2-1)(p^2-p)$.
Now, using the exact sequence from Lemma \ref{exact sequence}
$$1 \to \mathrm{Hom}(Z(G), \mathbb{C}^{\times}) 
\xrightarrow{\mathrm{tra}} 
\mathrm{H}^2(G_1, \mathbb{C}^{\times}) 
\xrightarrow{\mathrm{inf}} \mathrm{H}^2(G, \mathbb{C}^{\times}),$$ 
we observe that there are $(p^3-1)$ non-trivial $[\alpha]$ in $\operatorname{Im}\operatorname{(inf)}$.
Therefore, for each such $\alpha$, there is a $\beta \in Z^2(G_1, \mathbb{C}^{\times})$ such that $$\mathrm{Irr}^{\alpha}(G) \longleftrightarrow \bigcup_{\chi \in \mathrm{Hom}(Z(G),\C^{\times})} \mathrm{Irr}^{\beta\mathrm{tra}(\chi)}(G_1),$$ by Remark \ref{remark for bijective correspondence via inflation}(1).
If, for every non-trivial $[\alpha] \in \operatorname{Im} \operatorname{(inf)}$, there exist at least $p-1$ distinct $\chi$ such that each $[\beta \mathrm{tra}(\chi)]\in X$, then $(p^3-1)(p-1)\leq |X|$, contradicting to the fact $|X| \leq (p^2-1)(p^2-p)$. 
Hence, there is an $\alpha \in \mathrm{Z}^2(G, \C^{\times})$ and $\chi_1, \chi_2 \in \mathrm{Hom}(Z(G), \mathbb{C}^{\times})$ with $\chi_1 \neq \chi_2$ such that
the elements in $\mathrm{Irr}^{\beta \mathrm{tra}(\chi_i)}\big(G_1\big)$ are of degree $p$ for  $i=1,2$.
Now, let $\tilde{\rho_i} \in \mathrm{Irr}^{\beta \mathrm{tra}(\chi_i)}\big(G_1\big)$ and $\rho_i\in \mathrm{Irr}^{\alpha}(G)$ correspond to $\tilde{\rho_i}$, via the correspondence in Remark \ref{remark for bijective correspondence via inflation}(2). Set $\rho= \rho_1 \oplus \rho_2$. By Theorem \ref{direct_sum}, $\rho$ is faithful. Hence, $\tau(G) \leq 2p$.  
Let $\psi$ be a faithful projective representation of $G$ of degree $\tau(G)$. Since $\delta(G) =p^2$, $\psi$ can not be an ordinary representation. Therefore, $\psi$ has to be a true projective representation and $\tau(G)= p$ or $2p$. By Remark \ref{some results on p-groups}$(iii)$, $\tau(G)> p$. Hence, the result follows. \\ 

$(xvi)$ Let $G = \Phi_8(32)$. By \cite[Theorem 3.1(4)(vi)]{kaur}, $\delta(G)= p^2$. Hence, the result follows from Lemma \ref{group with cyclic centre of order p}. \\

$(xvii)$ Let $G$ be one of the groups $\Phi_{9}(2111)a$ or $\Phi_{9}(2111)b_r$. Since $\delta(G)= p$ (by \cite[Theorem 3.1(4)(vii)]{kaur}), $\tau(G) \leq p+1$ due to Lemma \ref{relation}. The result is a consequence of Remark \ref{some results on p-groups}$(iii)$. \\

$(xviii)$ Let $G= \Phi_{9}(1^{5})$. Then, $B_{G}=\{H \mid |H|= p^6, H \in \Phi_{35}, \Phi_{36} \text{ or } \Phi_{37}\}$. If $H \in \Phi_{35}$ with $|H|= p^6$, then $\delta_{irr}(H)= p$ in view of Theorem \ref{embedding degree of a p group with cyclic centre}. Similarly, if $H \in \Phi_{36}$ or $\Phi_{37}$ with $|H|= p^6$, then $\delta_{irr}(H)= p^2$. Hence, $\tau_{irr}(G)= p$ due to Corollary \ref{Elementary abelian Schur multiplier}$(i)$. By Remark \ref{some results on p-groups}$(ii)$, $\tau(G)= p$. \\

$(xix)$ Suppose $G = \Phi_{10}(2111)a_r$ or $\Phi_{10}(2111)b_r$. 
As $M(G) \cong \Z/p\Z$, we have $A.G =G$ or $G^*$. Since $cd(G)=\{1,p,p^2\}$ (by \cite[Lemma 5.8]{Prajapati}), any irreducible projective representation of $G$ is of degree $1, p$ or $p^2$. Then Theorem \ref{embedding degree of a p group with cyclic centre} and Theorem \ref{bijective correspondence} yield that $\delta(G^*)= p^2$. Also, as $\delta(G)= p^2$, $\tau(G) \geq p^2$, due to Theorem \ref{projective embedding degrees}.  

Let $\rho = \oplus_{i=1}^l\rho_i$, where $\rho_i \in \mathrm{Irr}^\alpha(G)$. Let $\tilde\rho_i$ be $\beta\mathrm{tra}(\chi_i)$ representations of $G/Z(G)$ corresponding to $\rho_i$ via the correspondence given in Remark \ref{remark for bijective correspondence via inflation}. As $\rho$ is faithful, by Theorem \ref{direct_sum}, $\chi_i \neq \chi_j$ for $i \neq j$ and $\oplus_{i=1}^l\chi_i$ is faithful on $Z(G)$. Since $\chi_i \oplus \chi_j$ is faithful on $Z(G)$, we have $l=2$. If $\operatorname{deg}(\rho_i)<p^2$ for $i=1,2$, then $\operatorname{deg}(\rho) \leq 2p$ which is not possible.
Hence, $\operatorname{deg}(\rho) \geq p^2+p$ if $\alpha$ is a non-trivial cocycle and $\operatorname{deg}(\rho) \geq p^2+1$ if $\alpha=1$. Since $\delta(G)= p^2$, the result follows from Lemma \ref{relation}.  \\


$(xx)$ Let $G= \Phi_{10}(1^5)$. Then, $B_{G}=\{H \mid |H|= p^6, H \in \Phi_{38} \text{ or } \Phi_{39}\}$. By arguments similar to those in the proof of (xviii), we get $\tau_{irr}(G)=p^2$. 
Now, from the proof of \cite[Lemma 4.2]{hatuip5}, it follows that the group
\begin{eqnarray*}
	G^* &=&\langle \alpha, \alpha_j, \beta_i
	\mid [\alpha_i,\alpha] = \alpha_{i+1},\ 
	[\alpha_1,\alpha_2] = \alpha_4 \beta_3,\ 
	[\alpha_4,\alpha] = \beta_1,\ [\alpha_1,\alpha_4] = \\ 
	&& [\alpha_3,\alpha_2] = \beta_2,\  [\alpha_1,\alpha_3] = \beta_1 \beta_2,\ \alpha^p = \alpha_{j}^p = \beta_i^p = 1,\ 1 \le i \le 3,\ 1 \le j \le 4 \rangle
\end{eqnarray*} 
is a representation group of $G$ of order $p^8$ and $G^*/Z({G}^*) = G$, where $Z({G}^*) = \langle \beta_1, \beta_2, \beta_3 \rangle$.  \\
Consider the abelian normal subgroup
$N = \langle \beta_1, \beta_2, \beta_3, \alpha_3, \alpha_4 \rangle$ of ${G}^*$ of order $p^5$. 
Let $\chi' \in \operatorname{Irr}(Z({G}^*))$ such that $\chi'(\beta_1) = \xi^{r}$, $\chi'(\beta_2) = \xi^{s}$, $\chi'(\beta_3)= \xi^q$, where $\xi$ is a primitive $p$th root of unity and $0 \le r, s, q \le (p-1)$ such that $r$, $s$, $q$ are  not all $0$.
We use Clifford's theory here to study $\operatorname{Irr}(G^* \mid \chi')$. 
Let $\chi \in \operatorname{Irr}(N)$ such that
$\chi|_{Z({G}^*)} = \chi'$ and $\chi(\alpha_4) = \xi^t$ where $0 < t \le (p-1)$. Now $g \in {G}^*$ can be written as $g = \alpha^j \alpha_1^k \alpha_2^l n'$ for some $n' \in N$ and any element of $N$ is of the form $n_1 = \alpha_3^m \alpha_4^n z$ for some $z \in Z({G}^*)$.
Therefore,
\[
\chi^{\alpha^j \alpha_1^k \alpha_2^l}(\alpha_3^m \alpha_4^n  z)
= \chi(\alpha_3^m \alpha_4^n z)
\iff
\chi(\alpha^j \alpha_1^k \alpha_2^l \alpha_3^m \alpha_4^n \alpha_2^{-l} \alpha_1^{-k} \alpha^{-j})
= \chi(\alpha_3^m \alpha_4^n).\]
Upon simplification,  we obtain that
$\alpha^j \alpha_1^k \alpha_2^l n' \in I_{{G}^*}(\chi)$
if and only if
\[
\chi(\beta_1^{km-jn+m\binom{p-j}{2}} \beta_2^{-lm+kn+km} \alpha_4^{-jm}) = 1
\text{ i.e., } 
\xi^{\,r(km-jn+m\binom{p-j}{2})+s(-lm+kn+km)-tjm} = 1.
\] Let $\Gamma \in \operatorname{Irr}(G^* \mid \chi)$. Now the following cases occur.

\textbf{Case (a)} Suppose $s\neq 0$. Then $I_{{G}^*}(\chi) = \langle N, \alpha\alpha_1^{s^{-1}r}\alpha_2^{-s^{-1}t+(s^{-1}r)^2+2s^{-1}r} \rangle$ and $\deg(\Gamma)=p^2$.

\textbf{Case (b)} Suppose $r\neq 0$ and $s =0$. Then $I_{{G}^*}(\chi) = \langle N, \alpha_2 \rangle$ and $\deg(\Gamma)=p^2$. 

\textbf{Case (c)} Suppose $r = 0$, $s= 0$ and $t \neq 0$. Then $I_{{G}^*}(\chi) = \langle N, \alpha_1, \alpha_2 \rangle$ is of order $p^7$. Since $\chi$ extends to $N_1= \langle N, \alpha_1  \rangle$, let $\tilde{\chi}: N_1 \to \C^\times$ be such that $\tilde{\chi}|_N = \chi$. Now,  $\tilde{\chi}^{(\alpha_2)}(\alpha_1)= \tilde{\chi}(\alpha_1)\tilde{\chi}(\alpha_4 \beta_3)^{-1}$. Therefore, $\chi$ extends to $I_{{G}^*}(\chi)$ if and only if $\chi(\alpha_4)= \chi(\beta_3)^{-1}$. 
Hence, $\deg(\Gamma) =p$ or $p^2$, and $\deg(\Gamma) =p$ if and only if $\chi$ extends to $I_{{G}^*}(\chi)$.
If $\deg(\Gamma)= p$, choosing $l_i= \alpha_3^i$, $i=1,2,..,p$, as left coset representatives of $I_{{G}^*}(\chi)$ in $G^*$, we get $\Gamma(\alpha_4)= \xi^{-q} I_{p \times p}$.

Now, suppose $\rho$ is a faithful $\alpha$-representation.
If $\rho$ is an ordinary representation, then  $\deg(\rho) \geq \delta(G) =p^2$. Otherwise, we have $\alpha= \operatorname{tra}(\chi')$ with $\chi'\neq 1$.
Let $\rho =\oplus_{i=1}^l\rho_i$, where $\rho_i \in \irr^\alpha(G)$ such that $\deg(\rho_i)=p$. Then, by Theorem \ref{restriction}, $\rho|_{Z(G)}$ is faithful. Let $\Gamma_i$ correspond to $\rho_i$, where the correspondence is given by Theorem \ref{bijective correspondence}. If $\deg(\Gamma_i) = p$, then 
by the above discussion,
$\Gamma_i(\alpha_4)=\xi^{-q} I_{p \times p}$ and in this case, $\rho|_{Z(G)}$ cannot be faithful, and so $\rho$ is not faithful. It says that $\rho$ must be an irreducible $\alpha$-representation.
Thus, $\tg= \ti =p^2$.

\end{proof}

We summarize our results on $\tau(G)$ and $\tau_{irr}(G)$, given in this section, in the following tables. If $\tau_{irr}(G)$ does not exist for a group $G$, then we denote this by writing $\tau_{irr}(G)=$ -.

\begin{table}[h!]
\centering
\small
\setlength{\extrarowheight}{2pt}
\begin{tabular}{|c|c|c|}
	\hline
	$G$ & $\tau(G)$ & $\tau_{irr}(G)$ \\ 
	\hline
	$\Phi_2(21)$ & $p+1$ & - \\ 
	$\Phi_2(1^3)$ & $p$ & $p$ \\ \
	$Q_8$ & 3 & - \\ 
	$D_8$ & 2 & 2 \\ 
	\hline
\end{tabular}
\caption{$\tau(G)$ and $\tau_{irr}(G)$ for non-abelian groups of order $p^3$}
\end{table}

\begin{table}[h!]
\centering
\small
\begin{tabular}{|c|c|c|c|}
	\hline
	Group ID & Group & $\tau(G)$ & $\tau_{irr}(G)$ \\
	\hline
	3  & $(\mathbb{Z}/2\mathbb{Z})^{2} \rtimes \mathbb{Z}/4\mathbb{Z}$ & 4 & 4 \\
	4  & $(\mathbb{Z}/4\mathbb{Z}) \rtimes (\mathbb{Z}/4\mathbb{Z})$ & 4 & - \\
	6  & $\mathbb{Z}/8\mathbb{Z} \rtimes_{5} \mathbb{Z}/2\mathbb{Z}$ & 3 & - \\
	7  & $D_{16}$ & 2 & 2 \\
	8  & $\mathbb{Z}/8\mathbb{Z} \rtimes_{3} \mathbb{Z}/2\mathbb{Z}$ & 3 & - \\
	9  & $Q_{16}$ & 2 & - \\
	11 & $D_{8} \times \mathbb{Z}/2\mathbb{Z}$ & 4 & 4 \\
	12 & $Q_{8} \times \mathbb{Z}/2\mathbb{Z}$ & 4 & - \\
	13 & $Q_{8} \rtimes \mathbb{Z}/2\mathbb{Z}$ & 3 & - \\
	\hline
\end{tabular}
\caption{$\tau(G)$ and $\tau_{irr}(G)$ for non-abelian groups of order $2^4$}
\end{table}

\begin{table}[!hbt]
\centering
\small
\begin{tabular}{|c|c|c|}
	\hline
	Group & $\tau(G)$ & $\tau_{irr}(G)$ \\
	\hline
	$\Phi_{2}(211)a$ & $p+2$ & - \\
	$\Phi_{2}(1^{4})$ & $p+2$ & $p^2$ \\
	$\Phi_{2}(31)$ & $p+1$ & - \\
	$\Phi_{2}(22)$ & $p+2$ & $p^2$ \\
	$\Phi_{2}(211)b$ & $p+1$ & -\\
	$\Phi_{2}(211)c$ & $p+2$ & - \\
	$\Phi_{3}(211)a$ & $p+1$ & -\\
	$\Phi_{3}(211)b_{r}$ & $p+1$ & -\\
	$\Phi_{3}(1^{4})$ & $p$ & $p$ \\
	\hline
\end{tabular}
\caption{$\tau(G)$ and $\tau_{irr}(G)$ for non-abelian groups of order $p^4$, $p \geq 3$} 
\end{table}

\begin{table}[H]
\centering
\small
\begin{tabular}{|c|c|c||c|c|c|}
\hline
$G$ & $\tau(G)$ & $\tau_{irr}(G)$ & $G$ & $\tau(G)$ & $\tau_{irr}(G)$  \\ 
\hline
$\Phi_2(311)a$ & $p+2$ & - & $\Phi_4(221)e$ & $2p$ & - \\ 
$\Phi_2(221)a$ & $p+3$ & - & $\Phi_4(221)f_0$ & $2p$ & - \\ 
$\Phi_2(221)b$ & $p+2$ & - & $\Phi_4(221)f_r$ & $2p$ & - \\ 
$\Phi_2(2111)a$ & $p+3$ & - & $\Phi_4(2111)a$ & $2p$ & - \\ 
$\Phi_2(2111)b$ & $p+2$ & - & $\Phi_4(2111)b$ & $2p$ & - \\ 
$\Phi_2(2111)c$ & $p+3$ & - & $\Phi_4(2111)c$ & $2p$ & - \\ 
$\Phi_2(2111)d$ & $p+2$ & - & $\Phi_4(1^5)$ & $2p$ & $p^2$ \\ 
$\Phi_2(1^5)$ & $p+3$ & $p^2$ & $\Phi_5(2111)$ & $2p$ & - \\ 
$\Phi_2(41)$ & $p+1$ & - & $\Phi_5(1^5)$ & $2p$ & - \\ 
$\Phi_2(32)a_1$ & $p+2$ & - & $\Phi_6(221)a$ & $2p+1$ & - \\ 
$\Phi_2(32)a_2$ & $p+2$ & - & $\Phi_6(221)b_r, \, r \neq \tfrac{1}{2}(p-1)$ & $2p+1$ & - \\ 
$\Phi_2(311)b$ & $p+1$ & - & $\Phi_6(221)b_{\tfrac{1}{2}(p-1)}$ & $2p+1$ & $p^2$ \\ 
$\Phi_2(311)c$ & $p+2$ & - & $\Phi_6(221)c_r$ & $2p+1$ & - \\ 
$\Phi_2(221)c$ & $p+2$ & $p^2$ & $\Phi_6(221)d_0$ & $2p+1$ & $p^2$ \\ 
$\Phi_2(221)d$ & $p+3$ & - & $\Phi_6(221)d_r$ & $2p+1$ & - \\ 
$\Phi_3(2111)a$ & $p+2$ & - & $\Phi_6(2111)a$ & $2p+1$ & - \\ 
$\Phi_3(2111)b_r$ & $p+2$ & - & $\Phi_6(2111)b_1$ & $2p+1$ & - \\ 
$\Phi_3(1^5)$ & $p+2$ & $p^2$ & $\Phi_6(2111)b_\nu$ & $2p+1$ & - \\ 
$\Phi_3(311)a$ & $p+1$ & - &  $\Phi_6(1^5)$ & $2p$  & $p^2$ \\ 
$\Phi_3(311)b_r$ & $p+1$ & - & $\Phi_7(2111)a$ & $2p$ & - \\ 
$\Phi_3(221)a$ & $p+2$ & - & $\Phi_7(2111)b_r$ & $2p$ & - \\ 
$\Phi_3(221)b_r$ & $p+2$ & $p^2$ & $\Phi_7(2111)c$ & $2p$ & - \\ 
$\Phi_3(2111)c$ & $p+1$ & - & $\Phi_7(1^5)$ & $2p$ & $p^2$ \\ 
$\Phi_3(2111)d$ & $p+2$ & - & $\Phi_8(32)$ & $p^2+1$ & - \\ 
$\Phi_3(2111)e$ & $p+2$ & - & $\Phi_9(2111)a$ & $p+1$ & - \\ 
$\Phi_4(221)a$ & $2p$ & - & $\Phi_9(2111)b_r$ & $p+1$ & - \\ 
$\Phi_4(221)b$ & $2p$ & $p^2$ & $\Phi_9(1^5)$ & $p$ & $p$ \\ 
$\Phi_4(221)c$ & $2p$ & - & $\Phi_{10}(2111)a_r$ & $p^2+1$ & - \\ 
$\Phi_4(221)d_r, \, r \neq \tfrac{1}{2}(p-1)$ & $2p$ & - & $\Phi_{10}(2111)b_r$ & $p^2+1$ & -  \\ 
$\Phi_4(221)d_{\tfrac{1}{2}(p-1)}$ & $2p$ & - &  $\Phi_{10}(1^5)$ & $p^2$ & $p^2$ \\ 
\hline
\end{tabular}
\caption{$\tau(G)$ and $\tau_{irr}(G)$ for non-abelian groups of order $p^5$, $p \geq 5$.}
\end{table}

\section*{Acknowledgements}
The authors acknowledge the support of the National Institute of Science Education and Research (NISER), Bhubaneswar and Homi Bhabha National Institute (HBNI), Mumbai.



\bibliographystyle{plain}
\bibliography{ref}

\end{document}